\documentclass{amsart}

\usepackage{mathrsfs, color}
\usepackage{amsfonts}
\usepackage{amssymb,amsmath}
\usepackage{amsthm}
\usepackage[mathscr]{eucal}
\usepackage{xspace}

\usepackage{bbm}

\theoremstyle{plain}
\newtheorem{theorem}{Theorem}[section]

\newtheorem{lemma}{Lemma}[section]

\theoremstyle{definition}

\theoremstyle{remark}

\newtheorem{remark}{Remark}[section]
\numberwithin{equation}{section}

\newcommand{\Div}{\text{div}\hspace{0.5mm}}
\author[H. Frid]{Hermano Frid}
 \address{Instituto de Matem\'atica Pura e Aplicada - IMPA\\ Estrada Dona Castorina, 110\\
Rio de Janeiro, RJ, 22460-320, Brazil}
\email{hermano@impa.br}
\thanks{H.~Frid gratefully acknowledges the support from CNPq, through grant proc.~303950/2009-9, and FAPERJ, through grant E-26/103.019/2011.}

\author[D.~Marroquin]{Daniel R. Marroquin}
 \address{Instituto de Matem\'atica Pura e Aplicada - IMPA\\ Estrada Dona Castorina, 110\\
Rio de Janeiro, RJ, 22460-320, Brazil}
\thanks{D.~Marroquin thankfully acknowledges the support from CNPq, through grants proc. 150118/2018-0 and poc. 140375/2014-7}
\email{danielrm@impa.br}

\author[R.H.~Pan]{Ronghua Pan}
\address{School of Mathematics,
Georgia Institute of Technology,
686 Cherry Street, Skiles Building
Atlanta, GA 30332-0160 }
\thanks{R.~Pan is partly supported by the National Science Foundation under grant  DMS-1516415, and by National Natural Science Foundation of China under grant 11628103.. \\ This work is an outcome from the Special Visiting Researcher program of the project Science Without Borders of the Brazilian government under the proc.~no.~401233/2012.}
 \email{panrh@math.gatech.edu }

\title[ Modeling Aurora Type Phenomena by SW-LW Interactions]{Modeling Aurora Type Phenomena by \\ Short Wave-Long Wave Interactions\\ in Multi-Dimensional Large MHD Flows}

\subjclass[2010]{35Q35, 76A02, 76N10}

\keywords{Compressible MHD system, nonlinear Schr\"{o}dinger equations, time decay rate}

\begin{document}

\begin{abstract}
We establish the convergence of an approximation scheme to a model for aurora type phenomena.  The latter, mathematically, means a system describing the short wave-long wave (SW-LW) interactions  for compressible  magnetohydrodynamic (MHD) flows, introduced in a previous work, which presents short waves, governed by a nonlinear Schr\"o\-dinger (NLS) equation based on the Lagrangian coordinates of the fluid, and  long waves, governed by the  compressible MHD system. The NLS equation and the compressible MHD system are also explicitly coupled by an interaction potential in the NLS equation and an interaction surface force in the momentum equation of the MHD system, both multiplied by a small coefficient.  Since the compressible MHD flow is assumed to have large amplitude data, possibly  forming vacuum, the coefficient of the interaction terms may be taken as zero, due to the large difference in scale between the two types of waves. In this case, the whole coupling lies in the Lagrangian coordinates of the compressible MHD fluid upon which the NLS equation is formulated. However, due to the possible occurrence of vacuum, these Lagrangian coordinates are not well defined, and herein lies the importance of the approximation scheme.  The latter consists of a system that formally approximates the  SW-LW interaction system, including non-zero vanishing interaction coefficients,  together with an artificial viscosity in the continuity equation, an artificial energy balance term, an artificial pressure in the momentum equation and approximate Lagrangian coordinates, which circumvent the possible occurrence of vacuum.   We prove the convergence 
of the solutions of the approximation scheme to a solution of a system consisting of a NLS equation based on the coordinate system induced by the scheme, and a compressible MHD system. 

\end{abstract}

\maketitle

\section{Introduction}

The aim of this paper is to prove the convergence of an approximation scheme for a system of equations modeling short wave-long wave (SW-LW) interactions,  between the magnetohydrodynamics (MHD) equations and a nonlinear Schr\"{o}dinger (NLS) equation. In the model, the NLS equation is coupled to the MHD system along particle paths, meaning that the former is stated in a different coordinate system, namely, the Lagrangian coordinates of the fluid. As such the short wave may be regarded as a small perturbation that propagates along the streamlines of the magnetohydrodynamic medium. This consideration motivates us to view both the system and the approximation scheme as a model to describe and simulate aurora type phenomena. In this connection, we find in the exposition about auroras in the  Wikipedia the following paragraph (see also, e.g., \cite{SSIMOO}): ``Auroras are produced when the magnetosphere is sufficiently disturbed by the solar wind that the trajectories of charged particles in both solar wind and magnetospheric plasma, mainly in the form of electrons and protons, precipitate them into the upper atmosphere (thermosphere/exosphere) due to Earth's magnetic field, where their energy is lost.''

\medskip
The system we are concerned with is the following
\begin{flalign}
&\rho_t+\text{div}(\rho\mathbf{u})= 0,\label{E3rho}\\
&(\rho \mathbf{u})_t + \text{div}(\rho\mathbf{u}\otimes\mathbf{u}) + \nabla p = \alpha \nabla(g'(1/\rho)h(|\psi\circ \mathbf{Y}|^2)\\
&\qquad \qquad\qquad \qquad\qquad\qquad\qquad+(\nabla\times\mathbf{H})\times\mathbf{H} + \Div \mathbb{S},\nonumber\\
&\mathbf{H}_t - \nabla \times(\mathbf{u}\times\mathbf{H})= -\nabla\times(\nu\nabla\times \mathbf{H} ),\\
&\text{div}\hspace{1mm}\mathbf{H}=0,\\
&i\psi_t+\Delta_\mathbf{y} \psi = |\psi|^2\psi + \alpha g(v)h'(|\psi|^2)\psi. \label{E3psi}&&
\end{flalign}

The NLS equation and the compressible MHD system are also explicitly coupled by an interaction potential in the NSL equation, namely, $g(v)h'(|\psi|^2)$, and an interaction surface force in the momentum equation of the MHD system, namely, $\nabla(g'(1/\rho)h(|\psi\circ \mathbf{Y}|^2)$, both multiplied by a small coefficient $\alpha$.  Since the compressible MHD flow is assumed to have large amplitude data, possibly  forming vacuum, we might take $\alpha=0$, due to the large difference in scale between the two types of waves. In this case, the whole coupling lies in the Lagrangian coordinates of the compressible MHD fluid upon which the NLS equation is formulated. However, due to the possible occurrence of vacuum, these Lagrangian coordinates are not well defined, and herein lies the importance of the approximation scheme (see \eqref{app2rho}--\eqref{app2H'}, \eqref{uN}).  The latter consists of a system that formally approximates the  SW-LW interaction system \eqref{E3rho}--\eqref{E3psi}, including the interaction terms with $\alpha>0$,  together with an artificial viscosity in the continuity equation, an artificial energy balance term and an artificial pressure in the momentum equation and approximate Lagrangian coordinates, which circumvent the possible occurrence of vacuum.   The artificial viscosity in the continuity equation together with the artificial energy balance term and the artificial pressure are borrowed from the approximation scheme introduced by Feireisl \cite{Fe} for the Navier-Stokes equations. 

We prove the convergence 
of the approximate solutions, as  $\varepsilon,\alpha \to0$ and $N\to\infty$, to a solution of a system consisting of a NLS equation based on the coordinate system induced by the approximation scheme, and a compressible MHD system, where $\varepsilon$ is the artificial viscosity, $\alpha$ is the interaction coefficient and $N$ is the ``number of harmonics''  in the approximate Lagrangian velocity, $\mathbf{u}^N$.      The convergence of the scheme legitimizes the consideration of the induced coordinates as  generalized Lagrangian coordinates of the fluid. Since we address a boundary value problem on a bounded domain and a cubic NLS, our analysis is carried out in the two-dimensional space. Nevertheless, the same procedure could be carried out in $\mathbb{R}^3$ as long as we truncate the nonlinearity in the NLS equation, that is, we replace $|\psi|^2$ by, say, $\min\{|\psi|^2, R\}$, for some $R>0$ as large as we wish. Accordingly, the lower bound for the adiabatic exponent $\gamma$ instead of 1, in the 2D case, becomes $3/2$ in the 3D case.

In order to explain the terms in appearing in the equations \eqref{E3rho}--\eqref{E3psi} it is worth reviewing briefly the deduction of the model.

Consider the MHD equations describing the dynamics of a compressible isentropic conductive fluid in the presence of a magnetic field
\begin{flalign}
&\rho_t+\text{div}(\rho\mathbf{u})= 0,\label{MHDrho}\\
&(\rho \mathbf{u})_t + \text{div}(\rho\mathbf{u}\otimes\mathbf{u}) + \nabla p = (\nabla\times\mathbf{H})\times\mathbf{H} + \Div \mathbb{S} + \mathbf{F}_{\text{ext}},\label{MHDu}\\
&\mathbf{H}_t - \nabla \times(\mathbf{u}\times\mathbf{H})= -\nabla\times(\nu\nabla\times \mathbf{H} ),\\
&\text{div}\hspace{1mm}\mathbf{H}=0.&&
\end{flalign}

Here, $\rho\geq 0$ and $\mathbf{u}\in\mathbb{R}^3$ denote the fluid's density and velocity field, respectively, and $\mathbf{H}\in\mathbb{R}^3$ the magnetic field; $p$ denotes the pressure, $\mathbf{F}_{\text{ext}}$ accounts for possible external forces and $\mathbb{S}$ is the viscous stress tensor given by
\[
\mathbb{S}=\lambda (\text{div}\hspace{0.5mm}\mathbf{u})\text{Id} + \mu (\nabla \mathbf{u}+(\nabla \mathbf{u})^{\top}).
\]

The viscosity coefficients $\lambda$ and $\mu$ satisfy $2\mu+\lambda>0$ and $\mu>0$; $\nu>0$ is the magnetic diffusivity. 

The pressure, in general, depends on the density through a constitutive relation of the form
\[
p=p(\rho).
\]

The MHD system above is stated in the Eulerian coordinates. The dependent variables are functions of $(\mathbf{x},t)$ where the spatial variable $\mathbf{x}$ belongs to $\mathbb{R}^3$ (or to some domain contained in $\mathbb{R}^3$ that is occupied by the fluid). In the Eulerian variables the motion is described from an outsider point of view. 

As mentioned above, in the model, the nonlinear Schr\"{o}dinger equation is stated in the Lagrangian coordinates of the fluid. The Lagrangian description follows the flow, as if the observer is on a boat following the stream lines. Accordingly the Lagrangian coordinates are characterized by being constant along the streamlines of the fluid and the change of variables can be defined through the flux $\Phi$ associated to the fluid's velocity field $\mathbf{u}$, given by 
\begin{equation}
 \begin{cases}
  \frac{d\Phi}{dt}(t;\mathbf{x})=\mathbf{u}(t,\Phi(t;\mathbf{x})),\\
  \Phi(0;\mathbf{x})=\mathbf{x},
 \end{cases}\label{defflux}
\end{equation}
and the Lagrangian transformation $\mathbf{Y}(t,\mathbf{x})=(\mathbf{y}(\mathbf{x},t),t)$ can be defined by the relation
\begin{equation}
\mathbf{y}(t,\Phi(t;\mathbf{x}))=\mathbf{y}_0(\mathbf{x}),\label{lag}
\end{equation}
where the function $\mathbf{y}_0$ is a diffeomorphism which may be chosen conveniently according to the problem. In particular, from \eqref{defflux} we have that
the Jacobian $J_\Phi(t;\mathbf{x})=\det \left(\frac{\partial\Phi}{\partial\mathbf{x}}(t;\mathbf{x})\right)$ of the transformation $x\mapsto \Phi(t;\mathbf{x})$ satisfies
\begin{equation}
 \frac{dJ_\Phi}{dt}(t;\mathbf{x})=\text{div}\hspace{0.5mm}\mathbf{u}(t,\Phi(t;\mathbf{x}))J_\Phi(t;\mathbf{x}),\label{fluxo}
\end{equation}
\[
 J_\Phi(0;\mathbf{x})=1.
\]
Then, choosing in $\mathbf{R}^d$ (for $d=2$ or $3$)
\begin{equation}
 \mathbf{y}_0(\mathbf{x}):=\big(\int_0^{x_1}\rho_0(s,x_{2},x_3)ds,x_2,\cdots, x_d\big),\label{y0lagrang}
\end{equation}
where, $\rho_0$ is the initial density $\rho_0(\mathbf{x})=\rho(0,\mathbf{x})$, a straightforward calculation shows that \eqref{MHDrho}, \eqref{lag} and \eqref{fluxo} imply that the Jacobian of the change of variables defined as $J_{\mathbf{y}}(t;\mathbf{x}):=\det \left(\frac{\partial\mathbf{y}}{\partial\mathbf{z}}(t,\Phi(t;\mathbf{x}))\right)$ satisfies
\[
\frac{d}{dt}\left( \frac{\rho(t,\Phi(t;\mathbf{x}))}{J_{\mathbf{y}}(t;\mathbf{x})} \right)=0.
\]
That is,
\begin{equation}
\det \left( \frac{\partial \mathbf{y}}{\partial \mathbf{z}} (t,\mathbf{z}) \right)=\rho(t,\mathbf{z}),\label{lagrho}
\end{equation}
for all $(t,\mathbf{z})\in [0,\infty)\times \mathbf{R}^d$.

Note, that this relation implies that the Lagrangian transformation becomes singular in the presence of vacuum or concentration, that is, when the density vanishes or becomes infinity.

Next, we consider the following nonlinear Schr\"{o}dinger equation stated in the Lagragian coordinates
\begin{equation}
i\psi_t+\Delta_\mathbf{y} \psi = |\psi|^2\psi + G\psi, \label{Schro}
\end{equation}
where $\psi=\psi(\mathbf{y},t)$ is the complex valued wave function, $G$ is a real valued function corresponding to a potential due to external forces and $\mathbf{y}$ is the Lagrangian coordinate as defined above.

Finally, the Short Wave-Long Wave interactions are modelled by choosing the external force term $\mathbf{F}_{\text{ext}}$ in \eqref{MHDu} and the potential term $G$ in \eqref{Schro} as
\begin{equation}
 \mathbf{F}_{\text{ext}}=\alpha \nabla(g'(1/\rho)h(|\psi\circ \mathbf{Y}|^2)),\hspace{10mm}G=\alpha g(v)h'(|\psi|^2),\label{force}
\end{equation}
where the interaction coefficient $\alpha$ is a positive constant, $\mathbf{Y}(t,\mathbf{x})=(t,\mathbf{y}(t,\mathbf{x}))$ is the Lagrangian transformation as before, $v(t,\mathbf{y})$ is the \textit{specific volume} given by the relation
\begin{equation}
 v(t,\mathbf{y}(t,\mathbf{x}))=\frac{1}{\rho(t,\mathbf{x})},\label{SV}
\end{equation}
and $g,h:[0,\infty)\to[0,\infty)$ are nonnegative smooth functions.

The most important feature of this coupling is that it is endowed with an energy identity, which can be stated in differential form as
\begin{flalign}
 &\Big\{ \big( \rho( \frac{1}{2}|\mathbf{u}|^2 + e )+\frac{1}{2}|\mathbf{H}|^2 \big)_t + \big(\mu |\nabla_\mathbf{x} \mathbf{u}|^2+(\lambda+\mu)(\text{div}_\mathbf{x}\mathbf{u})^2 + \nu |\nabla_\mathbf{x}\times \mathbf{H}|^2\big) \label{difE}\\
 &\hspace{15mm} + \text{div}_{\mathbf{x}}\big( \mathbf{u}(\rho(\frac{1}{2}|\mathbf{u}|^2 + e) + p+\alpha g'(1/\rho)h(|\psi\circ \mathbf{Y}|^2)) \big)\nonumber \\
 &\hspace{30mm}- \text{div}_{\mathbf{x}}\big(\mathbb{S}\cdot \mathbf{u} +  (\mathbf{u}\times\mathbf{H})\times \mathbf{H} +  \mathbf{H} \times \nu(\nabla_{\mathbf{x}}\times \mathbf{H} )\big)\Big\} d\mathbf{x}\nonumber \\
 &\hspace{6mm}= -\Big\{ \big( \frac{1}{2}|\nabla_{\mathbf{y}}\psi(t,\mathbf{y})|^2 + \frac{1}{4}|\psi(t,\mathbf{y})|^4 + \alpha g(v(t,\mathbf{y})) h(|\psi(t,\mathbf{y})|^2 \big)_t \nonumber\\
 &\hspace{65mm}- \text{div}_{\mathbf{y}}(\overline{\psi}_t\nabla_{\mathbf{y}}\psi + \psi_t\nabla_{\mathbf{y}}\overline{\psi})\Big\}d\mathbf{y},&& \nonumber
\end{flalign}
where, $e=e(\rho)$ is the internal energy given by
\[
e(\rho):=\int^\rho \frac{p(s)}{s^2}ds.
\]

Indeed, the usual energy identity for the isentropic MHD equations reads
\begin{flalign}
 & \big( \rho( \frac{1}{2}|\mathbf{u}|^2 + e )+\frac{1}{2}|\mathbf{H}|^2 \big)_t + \big(\mu |\nabla \mathbf{u}|^2+(\lambda+\mu)(\Div\mathbf{u})^2 + \nu |\nabla\times \mathbf{H}|^2\big)  \label{MHDE}\\
 & + \text{div}\big( \mathbf{u}(\rho(\frac{1}{2}|\mathbf{u}|^2 + e) + p) \big)- \text{div}\big(\mathbb{S}\cdot \mathbf{u} +  (\mathbf{u}\times\mathbf{H})\times \mathbf{H} +  \mathbf{H} \times \nu(\nabla\times \mathbf{H} )\big)\nonumber\\
 &\hspace{6mm}= \mathbf{F}_{\text{ext}}\cdot \mathbf{u}.&&\nonumber
\end{flalign}
In our particular situation $\mathbf{F}_\text{ext}$ is given by \eqref{force}, so that
\[
 \mathbf{F}_{\text{ext}}\cdot \mathbf{u}=\alpha \text{div}(g'(1/\rho)h(|\psi\circ \mathbf{Y}|^2)\mathbf{u})-\alpha g'(1/\rho)h(|\psi\circ \mathbf{Y}|^2)\text{div}\mathbf{u}.
\]
Multiplying (\ref{E3rho}) by $-(1/\rho)\alpha g'(1/\rho)h(|\psi\circ\mathbf{Y}|^2)$ we deduce that
\[
 -\alpha g'(1/\rho)h(|\psi\circ\mathbf{Y}|^2)\text{div}\hspace{0.5mm}\mathbf{u} = -\alpha(g(1/\rho)_t+\mathbf{u}\cdot \nabla_{\mathbf{x}} g(1/\rho))h(|\psi\circ\mathbf{Y}|^2)\rho.
\]

Observe that from the definition of $\mathbf{Y}$ we have the conversion formula between Eulerian and Lagrangian coordinates:
\[
 \beta(t,\mathbf{y})_t=(\beta\circ\mathbf{Y}(t,\mathbf{x}))_t+\mathbf{u}\cdot\nabla_{\mathbf{x}}(\beta\circ\mathbf{Y}(t,\mathbf{x})),
\]
or synthetically,
\[
 \beta(t,\mathbf{y})_t=\beta_t(t,\mathbf{x})+\mathbf{u}\cdot\nabla_{\mathbf{x}}\beta(t,\mathbf{x}).
\]

Keeping in mind the previously deduced formula \eqref{lagrho} for the Jacobian of the Lagrangian transformation synthesized by the identity $d\mathbf{y}=\rho(t,\mathbf{x})d\mathbf{x}$, we multiply equation (\ref{E3psi}) by $\overline{\psi}_t$ (the complex conjugate of $\psi_t$), take real part and incorporate the definition of $G$ to obtain
\begin{align*}
 &-\alpha(g(1/\rho)_t+\mathbf{u}\cdot \nabla_{\mathbf{x}} g(1/\rho))h(|\psi\circ\mathbf{Y}|^2)\rho\hspace{0.5mm}d\mathbf{x} \\
 &\hspace{5mm}= -\alpha g(v(t,\mathbf{y}))_t \hspace{0.5mm}h(|\psi(t,\mathbf{y})|^2)\hspace{0.5mm}d\mathbf{y}\\
 &\hspace{5mm}=-\alpha \left\{\Big(g(v(t,\mathbf{y})) h(|\psi(t,\mathbf{y})|^2)\Big)_t - g(v(t,\mathbf{y})) \hspace{0.5mm}h(|\psi(t,\mathbf{y})|^2)_t \right\}d\mathbf{y}\\
 &\hspace{5mm}=-\Big\{ \big( \frac{1}{2}|\nabla_{\mathbf{y}}\psi(t,\mathbf{y})|^2 + \frac{1}{4}|\psi(t,\mathbf{y})|^4 + \alpha g(v(t,\mathbf{y})) h(|\psi(t,\mathbf{y})|^2 \big)_t \nonumber\\
 &\hspace{65mm}- \text{div}_{\mathbf{y}}(\overline{\psi}_t\nabla_{\mathbf{y}}\psi + \psi_t\nabla_{\mathbf{y}}\overline{\psi})\Big\}d\mathbf{y}.
\end{align*}

Putting all of this information together and replacing it in the energy identity (\ref{MHDE}) we arrive at \eqref{difE}.

In particular, under suitable integrability conditions, this identity yields an integral form of the conservation of energy:
\begin{align*}
 &\frac{d}{dt}\int \left( \rho\left( \frac{1}{2}|\mathbf{u}|^2 + e \right)+\frac{1}{2}|\mathbf{H}|^2 \right)d\mathbf{x} \\
 &+ \int \big(\mu |\nabla_\mathbf{x} \mathbf{u}|^2+(\lambda+\mu)(\text{div}_\mathbf{x}\mathbf{u})^2 + \nu |\nabla_\mathbf{x}\times \mathbf{H}|^2\big) d\mathbf{x}\\
 &+\frac{d}{dt}\int \left( \frac{1}{2} |\nabla_{\mathbf{y}}\psi(t,\mathbf{y})|^2  + \frac{1}{4}|\psi(t,\mathbf{y})|^4 + \alpha g(v(t,\mathbf{y})) h(|\psi(t,\mathbf{y})|^2) \right)d\mathbf{y}=0.
\end{align*}

This kind of coupling was first studied in 2011 by Dias and Frid in \cite{DFr} where, inspired by the work of Benney on short wave-long wave interactions in [5], they proposed a similar model consisting
of a coupling between the Navier-Stokes equations for a compressible isentropic (non-heat conductive) fluid and a nonlinear Schr\"{o}dinger equation, studying existence and uniqueness of global solutions and the problem of vanishing viscosity and interaction coefficient limit in the one space dimensional context.

Later, in 2014, Frid, Pan and Zhang included the thermal description and addressed the problem of global existence of smooth solutions to the Cauchy problem, when the initial data are smooth small perturbations of an equilibrium state; this time in the full three dimensional case (see \cite{FrPZ}).

More recently, in 2016, Frid, Jia and Pan, included the magnetic description, obtaining the model above, and showed existence, uniqueness and decay rates of smooth solutions for small initial data, also in the three-dimensional context (\cite{FrJP}).

Our main goal here is to study the initial-boundary value problem for the multidimensional case with large data. The main difficulty that we face in this setting is the possible occurrence of vacuum. As the Lagrangian transformation becomes singular in the presence of vacuum an effective coupling of the fluid equations with the nonlinear Schr\"{o}dinger equation can not be made in a straightforward way. 

In order to overcome these difficulties, we define the interaction through a regularized system that provides a good definition for an approximate Lagrangian coordinate. Then, after showing existence of solutions, we show convergence of the sequence to a solution of the limit decoupled system as the regularizing parameters vanish together with the interaction coefficient at a specific rate, thus making sense of the Short Wave-Long wave interactions in the limit process. Although the limit Schr\"{o}dinger equation, when the interaction coefficient $\alpha$ is equal to zero, is apparently decoupled from the MHD system, it is stated in a coordinate system associated to the limit velocity field through the limit process.

For simplicity, we focus on the isentropic case, that is, the case of a non heat-conductive fluid. 

Let us remark that the results that we present here hold in a smooth bounded open spatial domain in $\mathbb{R}^2$. The only restriction that does not allow us to proceed in the full three dimensional one comes from the lack of global solvability of the {\em cubic} nonlinear Schr\"{o\-}dinger equation in a bounded domain of $\mathbb{R}^3$. However, should this be shown to hold our methods can be adapted to the three dimensional case. We also remark that the same procedure could be carried out in $\mathbb{R}^3$ as long as we truncate the nonlinearity in the NLS equation, that is, we replace $|\psi|^2$ by, say, $\min\{|\psi|^2, R\}$, for some $R>0$ as large as we wish. Accordingly, the lower bound for the adiabatic exponent $\gamma$ instead of 1, in the 2D case, becomes $3/2$ in the 3D case.

\subsection{Regularized problem}\label{reg}

In this subsection we introduce our regularized system that allows us to model the short wave-long wave interactions allowing for large initial data. Unfortunately, as already mentioned, a technical difficulty, related to the global solvability of the  cubic nonlinear Schr\"{o}dinger equation in a bounded domain of $\mathbb{R}^3$, prevents us to proceed in the full three dimensional setting. Thus our analysis is limited to the two dimensional case. 

Let us consider the two-dimensional system for an isentropic magnetohydrodynamic flow. The two-dimensional MHD equations are deduced from the full three-dimensional ones under the assumption that all the involved functions are independent of the third variable. Accordingly, we assume that our state variables $\rho\geq 0$, $\mathbf{u}\in \mathbb{R}^3$ and $\mathbf{H}\in\mathbb{R}^3$ are functions of $(\mathbf{x},t)\in\Omega\times [0,T]$, with $\Omega$ a smooth bounded domain of $\mathbb{R}^2$ and $T>0$ arbitrary.  Appealing to some abuse of notation and keeping in mind that the partial derivatives of the involved functions with respect to the third variable are zero, we can write the two-dimensional system exactly as the full three dimensional one as follows
\begin{flalign}
 &\rho_t+\text{div}(\rho\mathbf{u})=0,&\label{MHD2rho}\\
 &(\rho \mathbf{u})_t + \Div(\rho\mathbf{u}\otimes\mathbf{u}) + \nabla p = \Div\mathbb{S} + \left(\nabla\times\mathbf{H} \right)\times\mathbf{H} +\mathbf{F}_{\text{ext}},&\label{MHD2u}\\
&\mathbf{H}_t+\text{curl}\left(\nu\text{curl}\left( \mathbf{H} \right)\right)=\text{curl}(\mathbf{u}\times\mathbf{H}),&\label{MHD2B}\\
&\text{div}\hspace{1mm}\mathbf{H}=0,&\label{MHD2B'}
%& i\hspace{0.5mm}\psi_t+\Delta_{\mathbf{y}}\psi=|\psi|^2\psi+\tilde{\alpha}g(v)h'(|\psi|^2)\psi.&
\end{flalign}
where, as before,
\[
\mathbb{S}=\lambda (\text{div}\hspace{0.5mm}\mathbf{u})\text{Id} + \mu (\nabla \mathbf{u}+(\nabla \mathbf{u})^{\top}).
\]

Regarding the pressure $p$, we assume that it is given by a $\gamma$-law, that is,
\[
p(\rho)=a\rho^\gamma,
\]
for some $a>0$ and $\gamma>1$. 

Since we allow for large initial data, we work with weak solutions. As a result, the Lagrangian transformation as defined before may become singular due to the possible occurrence of vacuum in finite time. 

In order to work around the lack of regularity of the density, we first add an artificial viscosity to the continuity equation \eqref{MHD2rho}. Fix $\varepsilon >0$ and $\delta>0$ and consider the following regularized system
\begin{flalign}
&\rho_t+\text{div}(\rho\mathbf{u})= \varepsilon \Delta \rho,\label{app2rho}\\
&(\rho \mathbf{u})_t + \text{div}(\rho\mathbf{u}\otimes\mathbf{u}) + \nabla (a\rho^\gamma+\delta\rho^\beta) + \varepsilon \nabla\mathbf{u}\cdot\nabla\rho\label{app2u} \\
&\hspace{45mm}=\Div\mathbb{S}+(\nabla\times\mathbf{H})\times\mathbf{H}+\mathbf{F}_{\text{ext}},\nonumber\\
&\mathbf{H}_t - \nabla \times(\mathbf{u}\times\mathbf{H})= -\nabla\times(\nu\nabla\times \mathbf{H} ),\label{app2H}\\
&\text{div}\hspace{1mm}\mathbf{H}=0.\label{app2H'}&&
\end{flalign}

Note that besides the artificial viscosity added to the continuity equation, two new terms appeared in the momentum equation (\ref{MHD2u}). The term $\delta \rho^\beta$, where $\beta >1$, acts as an artificial pressure and is intended to provide better estimates on the density, whereas the term $\varepsilon \nabla\mathbf{u}\cdot\nabla\rho$ is set to equate the unbalance in the energy estimates of the MHD equations caused by the introduction of the artificial viscosity. This approximate system resembles the one employed by Hu and Wang in \cite{HW} where they study the existence of weak solutions to the three dimensional MHD equations. A similar approximation was introduced by Feireisl, et al. in \cite{FeNP} in the study of the Navier-Stokes equations, who, in turn, followed the pioneering ideas by P.-L. Lions in \cite{Li}. Recall that $\varepsilon$ and $\delta$ are small positive constants and the analysis that we intend to develop will provide insights that justify the accuracy to which this regularized model approximates the desired SW-LW interaction.

Now, as it turns out, even in this regularized setting the velocity field might not be smooth enough to provide a good enough definition of Lagrangian transformation that we can work with. More specifically, in the present situation there is no a priori bound available for the Jacobian of the Lagrangian transformation, as it depends on the $L^\infty$ norm of $\Div \mathbf{u}$. For this reason we replace the velocity by a suitable smooth approximation $\mathbf{u}^N$ (which tends to $\mathbf{u}$ as $N\to \infty$) in the definition of the Lagrangian transformation. Thus obtaining an approximate Lagrangian coordinate defined as before with $\mathbf{u}$ replaced by $\mathbf{u}^N$.

In order to define such an approximation of the velocity we consider the following subspaces of $L^2(\Omega)$. For each $n\in\mathbb{N}$ consider the space $X_n\subseteq L^2(\Omega;\mathbb{R}^3)$ defined as
\[
X_n:=E_n\times E_n \times E_n,
\]
where, $E_n= \text{span}\{ \eta_j:j=1,...,n\}$ and $\eta_1,\eta_2,\cdots$ is the complete collection of normalized eigenvectors of the Laplacian with zero boundary condition in $\Omega$;
with respective projection
\[
P_n:L^2(\Omega)\to X_n,
\]
With this notation, given $N\in\mathbb{N}$ we define $\mathbf{u}^N$ as
\begin{equation}
\mathbf{u}^N=P_N \mathbf{u}.\label{uN}
\end{equation}

Note that for any $\mathbf{u}(\mathbf{x},t)$ that satisfies $\mathbf{u}(\cdot,t)\in L^2(\Omega)$ for a.e. $t$, $\mathbf{u}^N$ thus defined is smooth and can be written as
\begin{equation}
\mathbf{u}^N(\mathbf{x},t) = \sum_{j=1}^N \mathbf{u}_j^N(t) \eta_j(\mathbf{x}),\label{UNcoef}
\end{equation}
for some vector valued coefficients $\mathbf{u}_j^N(t)$, $j=1,\cdots,N$; and satisfies,
\begin{equation}
||\mathbf{u}_N(t)||_{L^2(\Omega)}=\left(\sum_{j=1}^N |\mathbf{u}_j^N(t)|^2\right)^{1/2}.
\end{equation}
In fact, in light of \eqref{UNcoef} we have that
\begin{equation}
||\nabla \mathbf{u}^N||_{L^\infty(\Omega)}\leq C_N ||\mathbf{u}^N||_{L^2(\Omega)}\leq C_N ||\mathbf{u}||_{L^2(\Omega)},\label{nablauN}
\end{equation}
where 
\begin{equation}
C_N:=N \max_{j=1,\cdots,N}||\nabla \eta_j||_{L^\infty(\Omega)}.\label{CN}
\end{equation}

With this in mind, we define the Lagrangian transformation $Y(t,\mathbf{x})=Y(t,\mathbf{y}(t,\mathbf{x}))$ through (\ref{defflux}), (\ref{lag}) with the fluid's velocity $\mathbf{u}$ replaced by $\mathbf{u}^N$. Recall that we have a certain flexibility in the choice of the function $\mathbf{y}_0$. In the previous Section we chose it in terms of the initial density as it yielded a convenient expression for the Jacobian of the Lagrangian transformation, namely \eqref{y0lagrang}.

In the present situation, however, as we allow for vacuum, even in the initial data, we go another direction and choose 
\[
\mathbf{y}_0(\mathbf{x})=\mathbf{x}.
\]

With this choice for the initial diffemorphism, we see that for every $t\geq 0$ the coordinate change is a diffeomorphism from $\Omega$ into itself as well, and this holds for any $N$. This is due to the zero boundary conditions satisfied by each approximate velocity field $\mathbf{u}^N$.

With these modifications we now have a smoothed Lagrangian coordinate. Nonetheless, with the new definition we lose relation (\ref{lagrho}) and instead we have
\begin{equation}
J_\mathbf{y}(t)=\exp\left[-\int_0^t \text{div}\hspace{1mm}u^N(s,\Phi(s,x))ds\right].
\end{equation}
Note that, by Poincar\'{e}'s inequality, \eqref{nablauN} implies
\begin{align}
\exp \left[-C_N\left(t+\int_0^t||u(s)||_{H_0^1(\Omega)}^2ds\right)\right]&\leq J_\mathbf{y}(t)\nonumber\\
&\leq \exp \left[C_N\left(t+\int_0^t||u(s)||_{H_0^1(\Omega)}^2ds\right)\right], \label{JY}
\end{align}
provided that $\mathbf{u}\in L^2(0,T;H_0^1(\Omega))$, which is to be expected for the kind of solutions that we work with.

Now that we have a Lagrangian coordinate we can talk about the SW-LW interactions. To that end, we consider the following nonlinear Schr\"{o}dinger equation stated in the newly defined Lagrangian coordinates
\begin{equation}
i\psi_t+\Delta_\mathbf{y} \psi = |\psi|^2\psi + G\psi,
\end{equation} 
where $\psi$ is the complex valued wave function and $G$ is a real valued function corresponding to a potential. In order to complete our regularized model we have to define the coupling terms through the external force term $\mathbf{F}_{\text{ext}}$ in \eqref{MHD2u} and the potential $G$. As before we choose $G$ as
\begin{equation}
G=\alpha g(v)h'(|\psi|^2).
\end{equation}
Regarding $\mathbf{F}_{\text{ext}}$ we choose 
\begin{equation}
\mathbf{F}_{\text{ext}}=\alpha\nabla\left(\frac{J_\mathbf{y}}{\rho}g'(1/\rho)h(|\psi\circ \mathbf{Y}|^2)\right),
\end{equation}
where, as before, $g,h:[0,\infty)\to[0,\infty)$ are nonnegative smooth functions and we demand that 
\begin{equation}
\begin{cases}
 g(0)=h(0)=0,\\
 \text{supp}\hspace{.5mm} g' \text{ compact in }(0,\infty),\\
 \text{supp}\hspace{.5mm} h' \text{ compact in } [0,\infty).
\end{cases}\label{gh}
\end{equation}

Note that this coincides with our previous choice \eqref{force} once we realize that in our original model we had that $J_\mathbf{y} = \rho$. Also observe that, although vacuum is permitted in our new model, the fact that $g$ is compactly supported in $(0,\infty)$ clarifies any ambiguity in the definition of $\mathbf{F}_{\text{ext}}$.

As a result we end up with the following system of equations:
\begin{flalign}
&\rho_t+\text{div}(\rho\mathbf{u})= \varepsilon \Delta \rho,\label{appE2rho}\\
&(\rho \mathbf{u})_t + \text{div}(\rho\mathbf{u}\otimes\mathbf{u}) + \nabla (a\rho^\gamma+\delta\rho^\beta) + \varepsilon \nabla\mathbf{u}\cdot\nabla\rho\label{appE2u} \\
&\hspace{20mm}=\nabla(\alpha \frac{J_\mathbf{y}}{\rho} g'(1/\rho)h(|\psi\circ \mathbf{Y}|^2)) +(\nabla\times\mathbf{H})\times\mathbf{H} + \Div\mathbb{S},\nonumber\\
&\mathbf{H}_t - \nabla \times(\mathbf{u}\times\mathbf{H})= -\nabla\times(\nu\nabla\times \mathbf{H} ),\label{appE2H}\\
&\text{div}\hspace{1mm}\mathbf{H}=0.\label{appE2H'}\\
&i\psi_t+\Delta_\mathbf{y} \psi = |\psi|^2\psi + \alpha g(v)h'(|\psi|^2)\psi, &&\label{appE2Scho}
\end{flalign}

Regarding this new system, we prove the existence of solutions on a time interval $[0,T^N]$, where $T^N$ depends on $\varepsilon$, $\alpha$ and $N$. After this, we show the convergence of the approximate solutions when the artificial viscosity $\varepsilon$ together with the interaction coefficients $\alpha$ tend to zero and as $N$ tends to infinity at a specific rate at which $T^N$ tends to infinity.  Then, we make $\delta$ tend to zero and show convergence (on an arbitrary time interval $[0,T]$) to a solution of the system formed by the MHD equations together with the decoupled nonlinear Schr\"{o}dinger equation. In other words, we find a solution to the limit decoupled system, consisting of the MHD equations and a nonlinear Schr\"{o}dinger equation, as the limit of a sequence of solutions of the regularized SW-LW interactions system.

As emphasized before, the proposed approximation scheme has the purpose to legitimize the coordinates of the limiting Schr\"{o}dinger equation to be considered as the Lagrangian coordinates of the fluid in a generalized sense.

We consider the initial-boundary value problem for system \eqref{appE2rho}-\eqref{appE2Scho} with initial data
\begin{equation}
(\rho,\rho \mathbf{u},\mathbf{H})(\mathbf{x},0)=(\rho_0,\mathbf{m}_0,\mathbf{H}_0)(\mathbf{x}),\hspace{10mm}\psi(y,0)=\psi_0(\mathbf{y}),\label{appE20}
\end{equation}
where $\mathbf{m}_0$ is the initial momentum. Again, as vacuum is possible, it is better to regard the initial data in terms of the momentum instead of the velocity field.

With respect to the boundary conditions we demand that
\begin{equation}
(\nabla\rho\cdot \mathbf{n},\mathbf{u},\mathbf{H})|_{\partial\Omega}=0,\hspace{10mm}\psi|_{\partial \Omega_\mathbf{y}}=0.\label{appE2bound}
\end{equation}

Note that a Neumann boundary condition was added for the density as a result of the introduction of the artificial viscosity in the continuity equation.

With this notation we can state our first main result as follows.

\begin{theorem}\label{regsys}
Let $T>0$ be given and $N\in\mathbb{N}$ be fixed. Suppose that the initial data is smooth and that
\begin{equation}
M_0^{-1}\leq \rho_0 \leq M_1,
\end{equation}
for some positive constants $M_0$ and $M_1$. Assume, further, that $\beta$ is big enough. Then, if $\varepsilon$ and $\alpha$ are small and satisfy $\frac{\varepsilon^2}{\alpha}\gg 1$, there exists a solution $(\rho,\mathbf{u},\mathbf{H},\psi)$ of \eqref{appE2rho}-\eqref{appE2Scho} with initial and boundary conditions \eqref{appE20}, \eqref{appE2bound}. Moreover there is some $1<r<2$, independent of $N$, $\varepsilon$, $\alpha$ and $\delta$ such that
\begin{enumerate}
\item $\rho$ is nonnegative and
 \begin{equation}
 \rho\in L^r(0,T;W^{2,r}(\Omega))\cap L^{\beta+1}(\Omega\times(0,T)),\hspace{10mm} \rho_t\in L^r(\Omega\times(0,T));
 \end{equation}
\item $\mathbf{u},\mathbf{H}\in L^2(0,T;H_0^1(\Omega))$; 
\item $\psi\in L^\infty(0,T;H_0^1(\Omega))$
\item the initial and boundary conditions are satisfied in the sense of traces.
\end{enumerate}

Furthermore, we have that
\begin{equation}
E(t)+\varepsilon \int_0^t\int_\Omega(a\gamma\rho^{\gamma-2}+\delta\beta\rho^{\beta-2})|\nabla \rho|^2d\mathbf{x}\hspace{.5mm}ds \leq E(0)+\varepsilon^{1/2} R,\label{regnergy}
\end{equation}
for a.e. $t\in [0,T]$, where
\begin{flalign}
&E(t)=\int_\Omega\left( \frac{1}{2}\rho|\mathbf{u}|^2 + \frac{a}{\gamma-1}\rho^\gamma +\frac{\delta}{\beta-1}\rho^\beta+\frac{1}{2}|\mathbf{H}|^2\right)d\mathbf{x}&\nonumber\\
&\hspace{20mm}+\int_{\Omega_\mathbf{y}}\left( \frac{1}{2}|\nabla_\mathbf{y}\psi|^2 + \frac{1}{4}|\psi|^4 + \alpha g(v)h(|\psi|^2)\right)d\mathbf{y}&\nonumber\\
&\hspace{30mm}+\int_0^t\int_\Omega (\mu |\nabla \mathbf{u}|^2 + (\lambda+\mu)(\text{div}\hspace{.5mm}\mathbf{u})^2+\nu|\nabla\mathbf{H}|^2)d\mathbf{x}ds,\label{defregEt}
\end{flalign}
and
\[
R:=\varepsilon ||\rho_0||_{W^{2,r}(\Omega)} + ||\rho_0||_{H^1(\Omega)}^2 + E(0) + 1.
\]
\end{theorem}

Let us make some remarks on the statement of this theorem. First, the largeness assumed on $\beta$ is to be understood in the following sense. Theorem \ref{regsys} holds, as will be shown later, with $r\in (1,2)$ as long as $\beta > \max\{ \tfrac{2r}{2-r},\tfrac{2r}{r-1} \}$. Second, Theorem \ref{regsys} does not actually assert the existence of global solutions to the regularized SW-LW interactions. It affirms that given a prefixed $T>0$, there is a solution in the time interval $[0,T]$ satisfying \eqref{regnergy} as long as $\frac{\varepsilon^2}{\alpha}$ is big enough. Remember that $\varepsilon$ is an artificial small parameter we introduced in order to regularize the continuity equation. The reason for this hypothesis is to control uniformly in $N$ the Jacobian of the regularized Lagrangian transformation (which may explode as $N\to \infty$). More specifically, we are going to show that $\eqref{regnergy}$ holds as long as $T\leq T^N$, where $T^N=T^N(\alpha,\varepsilon)$ is defined in terms of $C^N$ from \eqref{nablauN} as
\begin{equation}
T^N:=\frac{1}{C_N}\log\left( \frac{\varepsilon^2}{\alpha} \right)-\frac{1}{\mu}(E(0)+\varepsilon^{1/2} R),\label{TN}
\end{equation}
whenever the right hand side is positive, which is the case, in particular, for $\alpha$ and $\varepsilon$ small enough satisfying $\frac{\varepsilon^2}{\alpha}\gg 1$.

We intend to analyze convergence of solutions to the regularized system as $(\varepsilon,\alpha,N)\to(0,0,\infty)$ and we do it based on the energy estimate \eqref{regnergy}. Thus, if we are looking for convergence to a {\bf global} solution of the limit problem we simply have to ensure that this $T^N$ covers any given bounded interval for big enough $N$ and small enough $\varepsilon$ and $\alpha$. This is the case if, for instance, we take the limit $(\varepsilon,\alpha,N)\to(0,0,\infty)$ at any rate that satisfies
\begin{equation}
\left( \frac{\varepsilon^2}{\alpha}\right)^{1/C_N}\to\infty.
\end{equation}

\bigskip
 Consider the limit problem
\begin{flalign}
&\rho_t+\text{div}(\rho\mathbf{u})= 0\label {deltaE2rho}\\
&(\rho \mathbf{u})_t + \text{div}(\rho\mathbf{u}\otimes\mathbf{u}) + \nabla (a\rho^\gamma+\delta\rho^\beta)\label{deltaE2u} \\
&\hspace{30mm}=(\nabla\times\mathbf{H})\times\mathbf{H} + \mu \Delta\mathbf{u} + (\lambda + \mu)\nabla(\text{div}\mathbf{u}),\nonumber\\
&\mathbf{H}_t - \nabla \times(\mathbf{u}\times\mathbf{H})= -\nabla\times(\nu\nabla\times \mathbf{H} ),\label{deltaE2H}\\
&\text{div}\hspace{1mm}\mathbf{H}=0.\label{deltaE2H'}\\
&i\psi_t+\Delta_\mathbf{y} \psi = |\psi|^2\psi, &&\label{deltaE2Scho}
\end{flalign}
with initial and boundary conditions \eqref{appE20} and 
\begin{equation}
(\mathbf{u},\mathbf{H})|_{\partial\Omega}=0,\hspace{10mm}\psi|_{\partial \Omega_{\mathbf{y}}}=0.\label{deltabound}
\end{equation}

Concerning the convergence of the approximate solutions provided by Theorem~\ref{regsys} we establish the following. 

\begin{theorem}\label{teoeps0}
Let $(\rho_\varepsilon,\mathbf{u}_\varepsilon,\mathbf{H}_\varepsilon,\psi_\varepsilon)$ be the solution of the regularized system \eqref{appE2rho}-\eqref{appE2Scho} provided by Theorem \ref{regsys}.  Then, there is a subsequence (not relabelled) that converges to a global weak solution $(\rho,\mathbf{u},\mathbf{H},\psi)$ of system \eqref{deltaE2rho}-\eqref{deltaE2Scho}, where the initial and boundary conditions are satisfied in the sense of distributions, as $(\varepsilon,\alpha,N)\to (0,0,\infty)$ provided that
\begin{equation}
\left( \frac{\varepsilon^2}{\alpha} \right)^{1/C_N}\to\infty,\label{epsalphN}
\end{equation}
where $C_N$ is given by \eqref{CN}.

Moreover, the density $\rho$ is nonnegative and satisfies equation \eqref{deltaE2rho} in the sense of renormalized solutions, meaning that 
\begin{equation}
B(\rho)_t + \text{div}(B(\rho)\mathbf{u})+b(\rho)\text{div}\mathbf{u} = 0,\label{renormalized}
\end{equation}
is satisfied also in the sense of distributions, for any functions 
\begin{equation}
B\in C[0,\infty)\cap C^1(0,\infty),\hspace{4mm}b\in C[0,\infty), \text{ bounded on }[0,\infty),\hspace{4mm}B(0)=b(0)=0,\label{renormB}
\end{equation}
satisfying
\begin{equation}
b(z)=B'(z)z-B(z).\label{renormBb}
\end{equation}

Furthermore, we have that 
\begin{align}
&\rho_\varepsilon\to \rho \text{ weakly-* in }L^\infty(0,T;L^\beta(\Omega))\label{rhoepstorho}\\
&\mathbf{u}_\varepsilon\to \mathbf{u} \text{ weakly in }L^2(0,T;H_0^1(\Omega))\label{uepstou}\\
&\mathbf{H}_\varepsilon\to \mathbf{H} \text{ strongly in }L^2(\Omega\times(0,T))\label{HepstoH}\\
&\hspace{30mm} \text{ and weakly-* in }L^2(0,T;H^1(\Omega))\cap L^\infty(0,T;L^2(\Omega))\nonumber\\
&\psi_\varepsilon \to \psi \text{ strongly in }C(0,T;L^2(\Omega)) \text{ and weakly-* in } L^\infty(0,T;H_0^1(\Omega)),\label{psiepstopsi}
\end{align}
\begin{equation}
\rho_\varepsilon\to\rho \text{ in }C([0,T];L^1(\Omega))\label{rhoepstorhoCL1}
\end{equation}
and
\begin{equation}
a\rho_\varepsilon^\gamma + \delta\rho_\varepsilon^\beta \to a\rho^\gamma + \delta\rho^\beta, \label{rhoepsstrongtorho}
\end{equation}
in the sense of distributions, are satisfied along with the energy inequality
\begin{equation}
E_\delta(t)\leq E_\delta(0),\label{Edeltaleq}
\end{equation}
for a.e. $t\in(0,T)$ where, 
\begin{flalign}
&E_\delta(t)=\int_\Omega\left( \frac{1}{2}\rho|\mathbf{u}|^2 + \frac{a}{\gamma-1}\rho^\gamma +\frac{\delta}{\beta-1}\rho^\beta+\frac{1}{2}|\mathbf{H}|^2\right)d\mathbf{x} \label{EtleqE0}\\
&\hspace{15mm}+\int_{\Omega_\mathbf{y}}\left( \frac{1}{2}|\nabla_\mathbf{y}\psi|^2 + \frac{1}{4}|\psi|^4 \right)d\mathbf{y}&\nonumber\\
&\hspace{25mm}+\int_0^t\int_\Omega (\mu |\nabla \mathbf{u}|^2 + (\lambda+\mu)(\text{div}\hspace{.5mm}\mathbf{u})^2+\nu|\nabla\mathbf{H}|^2)d\mathbf{x}ds.\nonumber
\end{flalign} 
\end{theorem}

\bigskip
Consider also the limit problem as $\delta\to0$.
\begin{flalign}
&\rho_t+\text{div}(\rho\mathbf{u})= 0,\label{finalrho}\\
&(\rho \mathbf{u})_t + \text{div}(\rho\mathbf{u}\otimes\mathbf{u}) + \nabla (a\rho^\gamma)\label{finalu}\\
&\qquad\qquad =(\nabla\times\mathbf{H})\times\mathbf{H} + \text{div}\big( \lambda (\text{div}\mathbf{u})\text{Id} + \mu (\nabla\mathbf{u}+(\nabla\mathbf{u})^\top) \big),
\nonumber\\
&\mathbf{H}_t - \nabla \times(\mathbf{u}\times\mathbf{H})= -\nabla\times(\nu\nabla\times \mathbf{H} ),\label{finalH}\\
&\text{div}\hspace{1mm}\mathbf{H}=0.\label{finalH'}\\
&i\psi_t + \Delta_\mathbf{y}\psi = |\psi|^2\psi, &&\label{finalSch}
\end{flalign}
subject to initial and boundary conditions
\begin{equation}
(\rho,\rho \mathbf{u},\mathbf{H})(\mathbf{x},0)=(\rho_0,\mathbf{m}_0,\mathbf{H}_0)(x),\hspace{10mm} \psi(\mathbf{y},0)=\psi_0(\mathbf{y}),\label{final0}
\end{equation}
and
\begin{equation}
(\mathbf{u},\mathbf{H})|_{\partial\Omega}=0,\hspace{10mm}\psi|_{\partial\Omega_\mathbf{y}}=0,\label{finalbound}
\end{equation}
with initial data satisfying
\begin{equation}
\begin{cases}
&\rho_0\geq 0,\hspace{5mm}\rho_0\in L^\gamma(\Omega),\\
&\frac{|\mathbf{m}_0|}{\rho_0}\in L^1(\Omega),\\
&\mathbf{H}_0 \in L^2(\Omega),\\
&\psi_0\in H_0^1(\Omega).
\end{cases}\label{condsfinal0}
\end{equation}

Once  Theorem~\ref{teoeps0} has been proved,    by repeating line by line the arguments in \cite[Section 5]{HW} we  arrive at  our  final result.

\begin{theorem}\label{teofinal}
Let $(\rho_\delta,\mathbf{u}_\delta,\mathbf{H}_\delta,\psi_\delta)$ be the solution of the decoupled system \eqref{deltaE2rho}-\eqref{deltaE2Scho}, \eqref{deltabound} with initial data
\[
(\rho_\delta,\mathbf{u}_\delta,\mathbf{H}_\delta,\psi_\delta)|_{t=0}=(\rho_{0\delta},\mathbf{u}_{0\delta},\mathbf{H}_{0\delta},\psi_{0\delta})
\]
provided by Theorem \ref{teoeps0}. Then, as $\delta\to 0$ we have that 
\begin{align}
&\rho_\delta\to\rho, \text{ weakly-* in }L^\infty(0,T;L^\gamma(\Omega)) \text{ and strongly in }C([0,T];L_{weak}^\gamma(\Omega)),\label{rhodeltatorho}\\
&\mathbf{u}_\delta \to\mathbf{u} \text{ weakly in }L^2(0,T;H_0^1(\Omega)),\\
&\mathbf{H}_\delta\to\mathbf{H} \text{ weakly in }L^2(0,T;H_0^1(\Omega)) \text{ and strongly in }C([0,T];L_{weak}^2(\Omega)),\\
&\psi_\delta\to\psi \text{ strongly in }C([0,T];L^4(\Omega)) \text{ and weakly-* in }L^\infty(0,T;H_0^1(\Omega)),\label{psideltatopsi}
\end{align}
subject to a subsequence as the case may be, where $(\rho,\mathbf{u},\mathbf{H},\psi)$ is a global weak solution of \eqref{finalrho}-\eqref{finalSch} with initial data \eqref{final0} satisfying \eqref{condsfinal0} and boundary conditions \eqref{finalbound}, satisfied in the sense of distributions. In fact we have that
\begin{equation}
\rho_\delta\to\rho, \text{ in }C([0,T];L^1(\Omega))
\end{equation}

Moreover, $\rho$ solves \eqref{finalrho} in the sense of renormalized solutions, meaning that  \eqref{renormalized} is satisfied in the sense of distributions with $B$ and $b$ as in \eqref{renormB} and \eqref{renormBb}.

Furthermore, we have that
\begin{equation}
E(t)\leq E(0),\label{finiteenergy}
\end{equation}
for a.e. $t$ with
\begin{flalign}
&E(t)=\int_\Omega\left( \frac{1}{2}\rho|\mathbf{u}|^2 + \frac{a}{\gamma-1}\rho^\gamma +\frac{1}{2}|\mathbf{H}|^2\right)d\mathbf{x}&\label{Et}\\
&\hspace{25mm}+\int_0^t\int_\Omega (\mu |\nabla \mathbf{u}|^2 + (\lambda+\mu)(\text{div}\hspace{.5mm}\mathbf{u})^2+\nu|\nabla\mathbf{H}|^2)d\mathbf{x}ds,\nonumber
\end{flalign} 
and,
\begin{equation}
\int_{\Omega_\mathbf{y}}\left( \frac{1}{2}|\nabla_\mathbf{y}\psi|^2 + \frac{1}{4}|\psi|^4 \right)d\mathbf{y}=\int_{\Omega_\mathbf{y}}\left( \frac{1}{2}|\nabla_\mathbf{y}\psi_0|^2 + \frac{1}{4}|\psi_0|^4 \right)d\mathbf{y},\label{consenergypsi}
\end{equation}
also for a.e. $t$.
\end{theorem}

The rest of this paper is organized as follows. In Section \ref{firstapp} we show existence of solutions to the regularized problem. Finally Sections \ref{limit1} and \ref{limit2} are devoted to the convergence of the sequence of solutions to the approximate problem to a solution of the limit problem.
As emphasized above, the proposed approximation scheme has the purpose to
legitimize the coordinates of the limiting Schr\"odinger equation to be considered as the
Lagrangian coordinates of the fluid in a generalized sense.

\section{Solutions to the regularized system}\label{firstapp}

The proof of  Theorem~\ref{regsys} consists in a Faedo-Galerkin method. We are going to apply Shauder's fixed point theorem in the finite-dimensional space $X_n$ in order to solve the momentum equation, having solved all the other equations in terms of the velocity. This provides a local approximate solution of the regularized system. Then, we deduce an energy estimate, corresponding to \eqref{regnergy}, that allows us to extend the local approximate solutions to the time interval $[0,T^N]$. As mentioned before, our analysis is based on the work by Hu and Wang in \cite{HW} in the study of the multidimensional MHD equations and also on the work by Feireisl, et al. in \cite{FeNP} and the work of P.-L. Lions in \cite{Li} in the study of the Navier-Stokes equations, although we had to develop new estimates in order to include the SW-LW interactions.

The rest of this Section is devoted to the proof of Theorem~\ref{regsys}.

\subsection{Approximate solutions, Faedo-Galerkin scheme}

Let us now fix $\varepsilon$, $\alpha$, $\delta$, $\beta$ and $N$ as in the statement of Theorem \ref{regsys}. For each $n\in\mathbb{N}$, we consider the space $X_n$ as defined before. We are going to apply Schauder's fixed point theorem in order to find a function $\mathbf{u}_n\in C(0,T;X_n)$ that satisfies equation \eqref{appE2u} in an approximate way. In order to achieve this, we must first show that given a function $\mathbf{u}\in X_n$ all the other equations \eqref{appE2rho}, \eqref{appE2H}, \eqref{appE2H'} and \eqref{appE2Scho} can be solved in terms of it.

Let us begin with the solvability of the continuity equation in terms of the velocity. Specifically, we consider the problem
\begin{equation}
\begin{cases}
\rho_t+\text{div}(\rho\mathbf{u})= \varepsilon \Delta \rho, &\text{ on } \Omega\times(0,T)\\
\nabla \rho\cdot \mathbf{n}=0, & \text{ on }\partial \Omega\\
\rho=\rho_0, &\text{ on }\Omega\times\{t=0\}.
\end{cases}\label{rhotermsofu}
\end{equation}

\begin{lemma}\label{lemmarhotermsu}
Let $\rho_0\in C^{2+\zeta}(\Omega)$, $\zeta>0$ and $u\in C([0,T];C_0^2(\overline{\Omega}))$ be given. Assume, further, that $\nabla \rho_0 \cdot \mathbf{n}=0$ on $\partial \Omega$.

Then, problem \eqref{rhotermsofu} has a unique classical solution $\rho$ such that
\begin{equation}
\partial_t \rho\in C([0,T];C^{\zeta}(\overline{\Omega})),\hspace{10mm}\rho\in C([0,T];C^{2+\zeta}(\Omega)).
\end{equation}
Moreover, suppose that the initial function $\rho_0$ is positive and let 
\begin{equation*}
\mathbf{u}\to\rho[\mathbf{u}]
\end{equation*}
be the solution mapping which assigns to any $\mathbf{u}\in C([0,T];C_0^2(\Omega))$ the unique solution $\rho$ of \eqref{rhotermsofu}.

Then, this mapping takes bounded sets in the space $C([0,T];C_0^2(\Omega))$ into bounded sets in the space
\begin{equation*}
V:=\{ \partial_t \rho\in C([0,T];C^{\zeta}(\overline{\Omega})), \rho\in C([0,T];C^{2+\zeta}(\Omega))\}
\end{equation*}
and
\begin{equation*}
\mathbf{u}\in C([0,T];C_0^2(\Omega))\to \rho[u]\in C^1([0,T]\times\overline{\Omega})
\end{equation*}
is continuous.
\end{lemma}

For the proof of this Lemma, we refer to \cite[Proposition~7.1]{Fe} (cf. \cite[Lemma~2.2]{FeNP}). Let us point out that solutions of the parabolic problem \eqref{rhotermsofu} obey the maximum principle which implies that
\begin{align}
&(\inf_{\mathbf{x}\in\Omega}\rho_0(\mathbf{x},0))\exp \left( -\int_0^t||\text{div}\hspace{.5mm}\mathbf{u}||_{L^\infty(\Omega)}ds\right)\leq \rho (\mathbf{x},t) \nonumber\\
&\hspace{40mm}\leq (\sup_{\mathbf{x}\in\Omega}\rho_0(\mathbf{x},0))\exp \left( \int_0^t||\text{div}\hspace{.5mm}\mathbf{u}||_{L^\infty(\Omega)}ds\right),\label{maxprinc}
\end{align}
for all $t\in [0,T]$ and all $\mathbf{x}\in\Omega$.

We also have to consider the following problem for the magnetic field
\begin{equation}
\begin{cases}
\mathbf{H}_t - \nabla \times(\mathbf{u}\times\mathbf{H})= -\nabla\times(\nu\nabla\times \mathbf{H} ), &\text{ on } \Omega\times(0,T)\\
\text{div}\hspace{1mm}\mathbf{H}=0,&\text{ on } \Omega\times(0,T)\\
\mathbf{H}=0, &\text{ on }\partial \Omega\\
\mathbf{H}=\mathbf{H}_0, &\text{ on }\Omega\times\{ t=0\}.
\end{cases}\label{Htermsofu}
\end{equation}

Regarding this problem we have the following result as presented by Hu and Wang (see \cite[Lemma~3.2]{HW}):
\begin{lemma}\label{lemmaHtermsu}
Assume that $\mathbf{u}\in C([0,T];C_0^2(\overline{\Omega}))$ is given. Then, problem \eqref{Htermsofu} has a unique solution $\mathbf{H}$ that satisfies
\begin{equation}
\mathbf{H}\in L^2(0,T;H_0^1(\Omega))\cap L^\infty(0,T;L^2(\Omega)),
\end{equation}
which solves \eqref{Htermsofu} in the weak sense and satisfies the initial and boundary conditions in the sense of traces. Moreover, let
\[
\mathbf{u} \to \mathbf{H}[\mathbf{u}]
\]
be the solution operator which assigns to $\mathbf{u}\in C([0,T];C^2(\overline{\Omega}))$ the unique solution $\mathbf{H}$ of \eqref{Htermsofu}. Then, this mapping maps bounded sets in $C([0,T];C_0^2(\overline{\Omega}))$ into bounded subsets of
\[
Y:=L^2(0,T;H_0^1(\Omega))\cap L^\infty(0,T;L^2(\Omega)),
\]
and 
\[
\mathbf{u}\in C([0,T];C^2(\overline{\Omega})) \to \mathbf{H}\in Y
\]
is continuous.
\end{lemma}

Finally, we move on to the solvability of the nonlinear Schr\"{o}dinger equation in terms of $\mathbf{u}$. It is this issue that poses a restriction on the dimension of $\Omega$. To our knowledge, the global solvability of the nonlinear Schr\"{o}dinger equation on a bounded domain of $\mathbb{R}^d$ with large initial data is an open problem for $d>2$. In the two-dimensional case, however, we have the result by Brezis and Gallouet at hand (see \cite{BrGa}) whose proof we can addapt to our present situation.

Consider the following problem
\begin{equation}
\begin{cases}
i\psi_t+\Delta_\mathbf{y} \psi = |\psi|^2\psi + \alpha g(v)h'(|\psi|^2)\psi, &\text{ on } \Omega_\mathbf{y}\times(0,T)\\
\psi=0, &\text{ on }\partial \Omega_\mathbf{y}\\
\psi=\psi_0, &\text{ on }\Omega_\mathbf{y}\times\{ t=0\},
\end{cases}\label{Schotermsofu}
\end{equation}
where, $v=v[\mathbf{u}]$ is given by
\[
v(t,\mathbf{y}(t,\mathbf{x}))=\frac{1}{\rho[\mathbf{u}](t,\mathbf{x})},
\]
$\rho[u]$ is as in Lemma \ref{lemmarhotermsu} and $\mathbf{y}$ is the approximate Lagrangian coordinate associated to the approximate velocity field $\mathbf{u}^N$ . Then, we can prove the following.
\begin{lemma}\label{lemmapsitermsu}
Assume that $\psi_0\in H^2(\Omega_{\mathbf{y}})\cap H_0^1(\Omega_{\mathbf{y}})$ and $\mathbf{u}\in C([0,T];C_0^2(\overline{\Omega}))$ are given. Then, problem \eqref{Schotermsofu} has a unique solution $\psi$ that satisfies
\begin{equation}
\psi\in C([0,T];H^2(\Omega_{\mathbf{y}})\cap H_0^1(\Omega_{\mathbf{y}}))\cap C^1([0,T];L^2(\Omega_{\mathbf{y}})).
\end{equation}
Moreover, let
\[
\mathbf{u} \to \psi[\mathbf{u}]
\]
be the solution operator which assigns to $\mathbf{u}\in C([0,T];C^2(\overline{\Omega}))$ the unique solution $\psi$ of \eqref{Schotermsofu}. Then, this mapping maps bounded sets in $C([0,T];C_0^2(\overline{\Omega}))$ into bounded subsets of
\[
Z:= C(0,T;H_0^1(\Omega_\mathbf{y})\cap L^2(\Omega)) 
\]
and
\[
\mathbf{u}\in C([0,T];C^2(\overline{\Omega})) \to \psi\in Z
\]
is continuous.
\end{lemma}

As this result is not explicitly covered by Brezis and Gallouet's one, we prove it next using an adaptation of their proof. For this we need the following two preliminary results.

The first one is due to Brezis and Gallouet and reads as

\begin{lemma}\label{lemmaBG}
There is a constant $C>0$ depending only on $\Omega$ such that
\[
||\psi||_{L^\infty(\Omega)}\leq C\big(1+\sqrt{\log[1+||\psi||_{H^2(\Omega)}]}\big),
\]
for every $\psi\in H^2(\Omega)$ with $||\psi||_{H^1(\Omega)}\leq 1$.
\end{lemma}
We refer to \cite{BrGa} for the proof. 

The second preliminary result is a particular case of a well known theorem  due to Segal  in \cite{Se},  whose statement below is found in \cite{BrGa}. 

\begin{lemma}\label{Segal}
Assume $H$ is a Hilbert Space and $A:D(A)\subseteq H\to H$ is an $m$-accretive linear operator. Assume $F$ is a mapping from $D(A)$ into itself which is Lipschitz on every bounded subset of $D(A)$, the latter endowed with the graph norm $\|\psi\|_A:=\|\psi\|+\|A\psi\|$. 

Then, for every $\psi_0\in D(A)$ there exists a unique solution $\psi$ of the equation
\[
\begin{cases}
\frac{d\psi}{dt}+A\psi = F(\psi),\\
\psi(0)=\psi_0,
\end{cases}
\] 
defined for $t\in[0,T_{max})$ such that 
\[
\psi\in C^1([0,T_{max});H)\cap C([0,T_{max});D(A)),
\]
with the additional property that 
\[
\begin{cases}
\text{either }T_{max}=\infty,\\
\text{or }T_{max}<\infty \text{ and }\lim_{t\nearrow T_{max}} ||\psi||_A=\infty.
\end{cases}
\]
\end{lemma}

Let us also state the very useful Aubin-Lions lemma (see \cite{JLi, Au}) for future reference. (A version of) Aubin-Lions lemma may be stated as (see \cite{Si})
\begin{lemma}[Aubin-Lions Lemma]\label{AubinLions}
Let $X_0$, $X$ and $X_1$ be Banach spaces such that
\[
X_0\subset X \subset X_1
\]
Supose that $X_0$ is compactly embedded in $X$ and that $X$ is continuously embedded in $X_1$. For $1\leq \alpha_0,\alpha_1 \leq \infty$, let
\[
W:=\{ v\in L^{\alpha_0}(0,T;X_0),\frac{dv}{dt}\in L^{\alpha_1}(0,T;X_1) \},
\]
under the norm
\[
||v||_W=||v||_{L^{\alpha_0}(0,T;X_0)}+\left\Vert \frac{dv}{dt} \right\Vert_{L^{\alpha_1}(0,T;X_1)}.
\]
Then,
\begin{enumerate}
\item[(i)] If $\alpha_0<\infty$, then the embedding of $W$ into $L^{\alpha_0}(0,T;X)$ is compact;
\item[(ii)] If $\alpha_0=\infty$ and $\alpha_1>1$, then the embedding of $W$ into $C([0,T];X)$ is compact.
\end{enumerate}
\end{lemma}

With these tools at hand we can proceed with the proof of Lemma \ref{lemmapsitermsu}.

\begin{proof}[Proof of Lemma \ref{lemmapsitermsu}]
We want to solve the equation \eqref{Schotermsofu}. For this, following closely \cite{BrGa}, we apply Lemma \ref{Segal} with $H=L^2(\Omega_\mathbf{y})$, $A(\psi)=\tfrac{1}{i}\Delta_\mathbf{y}\psi$, $D(A)=H^2(\Omega_\mathbf{y})\cap H_0^1(\Omega_\mathbf{y})$ and 
\[
F(\psi)=\tfrac{1}{i}|\psi|^2\psi + \tfrac{\alpha}{i}g(v)h'(|\psi|^2)\psi.
\]

It is enough to show that $||\psi||_{H^2(\Omega_\mathbf{y})}$ remains bounded on every bounded interval. Fix $T>0$ and consider $\psi$ solving \eqref{Schotermsofu} on the time interval $[0,T)$.

First, Multiplying \eqref{Schotermsofu} by $\overline{\psi}$, taking imaginary part and integrating we have
\[
||\psi(t)||_{L^2(\Omega_\mathbf{y})}=||\psi_0||_{L^2(\Omega_\mathbf{y})}.
\]
Similarly, multiplying \eqref{Schotermsofu} by $\overline{\psi_t}$, taking real part and integrating we have
\begin{align}
&\frac{1}{2}\int_{\Omega_\mathbf{y}}|\nabla \psi|^2d\mathbf{y} + \frac{1}{4}\int_{\Omega_\mathbf{y}} |\psi|^4d\mathbf{y} \nonumber\\
&\hspace{5mm}= \frac{1}{2}\int_{\Omega_\mathbf{y}}|\nabla \psi_0|^2d\mathbf{y} + \frac{1}{4}\int_{\Omega_\mathbf{y}} |\psi_0|^4d\mathbf{y} + \int_0^t\int_{\Omega_\mathbf{y}}\alpha g(v)h(|\psi|^2)_t d\mathbf{y}ds. 
\end{align}

Now, 
\begin{align*}
\int_0^t\int_{\Omega_\mathbf{y}}\alpha g(v)h(|\psi|^2)_td\mathbf{y}ds=&\int_{\Omega_\mathbf{y}}\alpha g(v)h(|\psi|^2)d\mathbf{y}-\int_{\Omega_\mathbf{y}}\alpha g(v_0)h(|\psi_0|^2)d\mathbf{y}\\
&-\int_0^t\int_{\Omega_\mathbf{y}}\alpha g(v)_t h(|\psi|^2)d\mathbf{y}ds.
\end{align*}

Regarding the last term on the right hand side and using the definition of the Lagrangian transformation
\begin{align}
&\int_0^t\int_{\Omega_\mathbf{y}}\alpha g(v)_t h(|\psi|^2)d\mathbf{y}ds =\nonumber\\
&\hspace{30mm} \int_0^t\int_{\Omega_\mathbf{y}}\alpha g'(1/\rho)h(|\psi\circ \mathbf{Y}|^2)\left(\big(\frac{1}{\rho}\big)_t+\mathbf{u}^N\cdot\nabla\big(\frac{1}{\rho}\big)  \right)J_\mathbf{y} dx.\label{asd}
\end{align}

As $\mathbf{u}\in C([0,T];C_0^2(\overline{\Omega}))$, we have that $|J_\mathbf{y}|\leq C$ and using \eqref{gh} and Lemma \ref{lemmarhotermsu} we have that the right hand side of \eqref{asd} is bounded, that is
\[
\left| \int_0^t\int_{\Omega_\mathbf{y}}\alpha g(v)_t h(|\psi|^2)d\mathbf{y}ds \right|\leq C.
\]

This implies that
\begin{equation}
||\nabla \psi(t)||_{L^2(\Omega_\mathbf{y})}\leq C.\label{psiH1bound}
\end{equation}

Next, let $S(t)$ be the isometry group generated by $A$. Then,
\[
\psi(t)=S(t)\psi_0 + \tfrac{1}{i}\int_0^tS(t-s)\big( |\psi(s)|^2\psi(s) - \alpha g(v)h'(|\psi(s)|^2)\psi(s) \big)ds,
\]
and, so
\[
A\psi(t)=S(t)A\psi_0 + \tfrac{1}{i}\int_0^tS(t-s)A\left[\big( |\psi(s)|^2\psi(s) - \alpha g(v)h'(|\psi(s)|^2)\psi(s) \big)\right]ds.
\]

Consequently, 
\begin{align*}
||A\psi(t)||_{L^2(\Omega_\mathbf{y})}\leq &||A\psi_0||_{L^2(\Omega_\mathbf{y})} + \int_0^t||A[|\psi(s)|^2\psi(s)]|_{L^2(\Omega_\mathbf{y})}ds\\
 &+\alpha \int_0^t||A\big[g(v(s))h(|\psi(s)|^2)\psi(s) \big]||_{L^2(\Omega_\mathbf{y})}ds.
\end{align*}

Using \eqref{psiH1bound}, Lemma \ref{lemmaBG} can be used to show that
\begin{equation}
\int_0^t||A[|\psi(s)|^2\psi(s)]|_{L^2(\Omega_\mathbf{y})}ds\leq C\int_0^t||\psi(s)||_{H^2(\Omega_\mathbf{y})}\big( 1+\log[1+||\psi(s)||_{H^2(\Omega_\mathbf{y})}] \big)ds.\label{ineqBG}
\end{equation}
Indeed, observe that
\[
|D^2(|\psi|\psi)|\leq C(|\psi|^2|D^2\psi|+|\psi|\hspace{1mm}|\nabla\psi|^2),
\]
which implies
\[
||\hspace{1mm}|\psi|^2\psi||_{H^2(\Omega_\mathbf{y})}\leq C||\psi||_{L^\infty(\Omega_\mathbf{y})}^2||\psi||_{H^2(\Omega_\mathbf{y})} + C||\psi||_{L^\infty(\Omega_\mathbf{y})}||\psi||_{W^{1,4}(\Omega_\mathbf{y})}^2.
\]
But, Gagliardo-Nirenberg Inequality implies (recall that $\Omega\subseteq \mathbb{R}^2$)
\[
||\psi||_{W^{1,4}(\Omega_\mathbf{y})}\leq C||\psi||_{L^\infty(\Omega_\mathbf{y})}^{1/2}||\psi||_{H^2(\Omega_\mathbf{y})}^{1/2}.
\]
These two inequalities combined together with Lemma \ref{lemmaBG} imply \eqref{ineqBG}.

A similar argument shows that
\begin{align*}
&\int_0^t||A\big[g(v(s))h(|\psi(s)|^2)\psi(s) \big]||_{L^2(\Omega_\mathbf{y})}ds\\
&\hspace{30mm}\leq C+C\int_0^t||\psi(s)||_{H^2(\Omega_\mathbf{y})}\big( 1+\log[1+||\psi(s)||_{H^2(\Omega_\mathbf{y})}] \big)ds.
\end{align*}
Here we have used \eqref{psiH1bound} and Lemma \ref{lemmarhotermsu}.

Thus we conclude that
\begin{equation}
||\psi(t)||_{H^2(\Omega_\mathbf{y})}\leq C+C\int_0^t||\psi(s)||_{H^2(\Omega_\mathbf{y})}\big( 1+\log[1+||\psi(s)||_{H^2(\Omega_\mathbf{y})}] \big)ds.
\end{equation}
Denoting $G(t)$ the right hand side of this inequality we have that
\[
G'(t)\leq CG(t)(1+\log[1+G(t)]),
\]
which, implies that
\[
\frac{d}{dt}\log\big[ 1+\log[1+G(t)] \big]\leq C
\]

And hence we arrive at an estimate of the form
\[
||\psi(t)||_{H^2(\Omega_\mathbf{y})}\leq e^{b_1e^{b_2 t}},
\]
for some constants $b_1$ and $b_2$ and every $t\in [0,T)$. In particular
\[
||\psi(t)||_{H^2(\Omega_\mathbf{y})}\leq e^{b_1e^{b_2 T}}, \text{ for every }t\in[0,T).
\]
As this holds for every $T>0$ we conclude that $T_{max}=\infty$.

In order to conclude the proof we have to show the stated continuity of the map $\mathbf{u}\to \psi[\mathbf{u}]$. Let $\{\mathbf{u}_k \}_k$ be a sequence in $C([0,T];C^2(\overline{\Omega}))$ such that $\mathbf{u}_k\to \mathbf{u}_\infty\in C([0,T];C^2(\overline{\Omega}))$, and let $v_k=v[\mathbf{u}_k]$, $v_\infty=v[\mathbf{u}_\infty]$, $\psi_k=\psi[\mathbf{u}_k]$ and $\psi_\infty=\psi[\mathbf{u}_\infty]$. In light of Lemma \ref{lemmarhotermsu} and by the smoothness of the Lagrangian transformation we have that $v_k\to v_\infty$ in $C^1(\overline{\Omega}\times[0,T])$. Next, by Aubin-Lions lemma (Lemma \ref{AubinLions}) we have that there is a subsequence $\{\psi_{k_j}\}_j$ that converges in $C([0,T];H_0^1(\Omega))$ to a solution $\psi$ of the limit equation \eqref{Schotermsofu} with $v=v_\infty$. By uniqueness we have that $\psi=\psi[\mathbf{u}_\infty]$ and also that the whole sequence $\{\psi_k\}_k$ converges to $\psi$ in $C([0,T];H_0^1(\Omega))$, thus concluding the proof.
\end{proof}

Having these results we can apply the Faedo-Galerkin method in order to find solutions to the regularized system. First, for each $n\in \mathbb{N}$, we are going to look for a function $u_n$ that satisfies \eqref{appE2u} in an approximate way. Specifically, we demand that $u_n$ satisfies
\begin{flalign}
&\int_\Omega\rho_n \mathbf{u}_n \cdot\eta d\mathbf{x} - \int_\Omega \mathbf{m}_0 \cdot \eta d\mathbf{x}&\nonumber\\
&\hspace{2.mm}\qquad  + \int_0^t\int_\Omega \left(\text{div}(\rho_n\mathbf{u}_n\otimes\mathbf{u}_n) + \nabla (a\rho_n^\gamma+\delta\rho_n^\beta) + \varepsilon \nabla\mathbf{u}_n\cdot\nabla\rho_n\right)\cdot\eta d\mathbf{x}\hspace{.5mm}ds&\label{galerkin} \\
&\hspace{10.mm}\qquad= \int_0^t\int_\Omega \big(\nabla(\alpha \frac{J_\mathbf{y}}{\rho_n} g'(1/\rho_n)h(|\psi_n\circ \mathbf{Y}|^2)) +(\nabla\times\mathbf{H}_n)\times\mathbf{H}_n &\nonumber\\
&\hspace{20.mm}+ \mu \Delta\mathbf{u}_n + (\lambda + \mu)\nabla(\text{div}\mathbf{u}_n)\big)\cdot\eta d\mathbf{x}\hspace{.5mm}ds,&\nonumber
\end{flalign}
for any $t\in [0,T]$ and any $\eta\in X_n$, where $\rho_n=\rho[\mathbf{u}_n]$, $\mathbf{H}_n=\mathbf{H}[\mathbf{u}_n]$, $\psi_n=\psi[\mathbf{u}_n]$ and $\mathbf{Y}$ is Lagrangian transformation associated to the velocity field $u_n^N=P_N\hspace{.5mm}u_n$, with Jacobian $J_\mathbf{y}$. This formulation may be interpreted as a projection of equation \eqref{appE2u} onto the finite dimensional space $X_n$.

Let us rewrite this integral equation in a more suitable way. Given some function $\rho\in L^1(\Omega)$, consider the operator $\mathcal{M}[\rho]:X_n\to X_n^*$, where $X_n^*$ is the dual space of $X_n$, given by
\[
\langle \mathcal{M}[\rho]\mathbf{v},\mathbf{w}\rangle:=\int_\Omega \rho\mathbf{v}\cdot\mathbf{w}.
\]

Then, the operator $\mathcal{M}$ is invertible provided that $\rho$ is strictly positive on $\Omega$ and  the map $\rho\to\mathcal{M}^{-1}[\rho]$, mapping $L^1(\Omega)$ into $\mathcal{L}(X_n^*; X_n)$, satisfies
\begin{equation}
||\mathcal{M}[\rho]^{-1}||_{\mathcal{L}(X_n^*;X_n)}\leq \frac{1}{\inf_{\Omega}\rho}.\label{normM}
\end{equation}
Moreover, the identity
\[
\mathcal{M}[\rho^1]^{-1}-\mathcal{M}[\rho^2]^{-1}=\mathcal{M}[\rho^2]^{-1}\big( \mathcal{M}[\rho^2]-\mathcal{M}[\rho^1] \big) \mathcal{M}[\rho^1]^{-1},
\]
can be used to obtain
\begin{equation}
||\mathcal{M}[\rho^1]^{-1}-\mathcal{M}[\rho^2]^{-1}||_{\mathcal{L}(X_n^*;X_n)}\leq c(n,\underline{\rho})||\rho^1-\rho^2||_{L^1(\Omega)},\label{contM}
\end{equation}
for any $\rho^1$ and $\rho^2$ such that
\[
\inf_\Omega \rho^1, \inf_\Omega \rho^2 \geq \underline{\rho}.
\]

In connection with \eqref{galerkin} we also define the operator $\mathcal{N}:X_n\to X_n^*$ given by
\begin{flalign*}
&\langle \mathcal{N}[\mathbf{u}],\eta\rangle = -\int_\Omega \left(\text{div}(\rho\mathbf{u}\otimes\mathbf{u}) + \nabla (a\rho^\gamma+\delta\rho^\beta) + \varepsilon \nabla\mathbf{u}\cdot\nabla\rho\right)\cdot\eta d\mathbf{x}&\nonumber \\
&\hspace{30mm}+ \int_\Omega \big(\nabla(\alpha \frac{J_\mathbf{y}}{\rho} g'(1/\rho_n)h(|\psi|^2)) +(\nabla\times\mathbf{H})\times\mathbf{H} &\nonumber\\
&\hspace{65mm}+ \mu \Delta\mathbf{u} + (\lambda + \mu)\nabla(\text{div}\mathbf{u})\big)\cdot\eta d\mathbf{x},&
\end{flalign*}
with $\rho=\rho[\mathbf{u}]$, $\mathbf{H}=\mathbf{H}[\mathbf{u}]$ and $\psi=\psi[\mathbf{u}]$.

With this notation, identity \eqref{galerkin} can be rewritten as
\[
\mathbf{u}_n(t)=\mathcal{M}[\rho_n(t)]^{-1}\left(\mathbf{m}_0^* + \int_0^t \mathcal{N}[\mathbf{u}_n(s)] ds \right).
\]
This means that we are looking for a fixed point of the application $\mathcal{T}:C([0,T];X_n)\to C([0,T];X_n)$ given by
\[
\mathcal{T}[\mathbf{u}](t)=\mathcal{M}[\rho[\mathbf{u}](t)]^{-1}\left(\mathbf{m}_0^* + \int_0^t \mathcal{N}[\mathbf{u}(s)] ds \right).
\]

Using Lemmas \ref{lemmarhotermsu}, \ref{lemmaHtermsu} and \ref{lemmapsitermsu}, as well as \eqref{normM} and \eqref{contM} and Arzel\`{a}-Ascoli theorem it can be shown that $\mathcal{T}$ maps bounded sets in $C([0,T];X_n)$ into precompact sets in $C([0,T];X_n)$.

Moreover, define $\mathbf{u}_0\in X_n$ as being the only element in $X_n$ that satisfies
\[
\int_\Omega \rho_0\mathbf{u}_0\cdot\eta d\mathbf{x}=\int_\Omega \mathbf{m}_0\cdot\eta d\mathbf{x},\hspace{10mm}\text{for all }\eta\in X_n.
\]
Consider a ball $\mathcal{B}:=\{ \mathbf{v}\in C([0,T];X_n): \sup_{t\in[0,T]}||\mathbf{v}(t)-\mathbf{u}_0||_{X_n}\leq 1\}$. Then, $\mathcal{T}$ maps the ball $\mathcal{B}$ into itself, provided $T=T(n)$ is small enough. Consequently, Schauder's fixed point theorem guarantees the existence of at least one fixed point $\mathbf{u}_n$, $\mathbf{u}_n=\mathcal{T}[\mathbf{u}_n]$ which provides a solution to \eqref{galerkin}. 

Now, we want to find a solution to the regularized system as a limit of the sequence $\mathbf{u}_n$.  However, the approximate velocity field $\mathbf{u}_n$ is defined only on the time interval $[0,T(n)]$. Accordingly, we have to guarantee that this solution can be extended to a uniform over $n$ time interval $[0,T^*]$. In order to achieve this, we deduce next some a priori estimates on the fixed point $\mathbf{u}_n$ we found above that allow us to iterate the fixed point argument a finite number of times until we reach the whole time interval $[0,T^*]$.

In the case of the MHD system and in the case of the Navier Stokes system, the conservation of energy provides good enough global a priori estimates that guarantee boundedness of the fixed point globally in time. In our present situation, however, the short wave-long wave interaction turns the estimate more difficult as the energy of the system is not well balanced. As a consequence we do not obtain a global a priori estimate. Fortunately, we are able to bound from below the maximal time during which the estimates hold by some $T^N$ independent of $n$ that satisfies the properties stated in Theorem \ref{regsys}.

The a priori estimates are based on the usual energy estimates for the MHD equations, but rely on a bootstrap argument in order to accommodate the unbalance in the energy caused by the short wave-long wave interactions coupling terms. 

For convenience, we define $E_n(t)$ as in \eqref{defregEt} with $(\rho,\mathbf{u},\mathbf{H},\psi)$ replaced by $(\rho_n,\mathbf{u}_n,\mathbf{H}_n,\psi_n)$. That is
\begin{flalign}
&E_n(t)=\int_\Omega\left( \frac{1}{2}\rho_n|\mathbf{u}_n|^2 + \frac{a}{\gamma-1}\rho_n^\gamma +\frac{\delta}{\beta-1}\rho_n^\beta+\frac{1}{2}|\mathbf{H}_n|^2\right)d\mathbf{x}&\nonumber\\
&\hspace{20.mm}+\int_{\Omega_\mathbf{y}}\left( \frac{1}{2}|\nabla_\mathbf{y}\psi_n|^2 + \frac{1}{4}|\psi_n|^4 + \alpha g(v_n)h(|\psi_n|^2)\right)d\mathbf{y}&\nonumber\\
&\hspace{30.mm}+\int_0^t\int_\Omega (\mu |\nabla \mathbf{u}_n|^2 + (\lambda+\mu)(\text{div}\hspace{.5mm}\mathbf{u}_n)^2+\nu|\nabla\mathbf{H}_n|^2)d\mathbf{x}ds.
\end{flalign}

In the notation of Theorem \ref{regsys} we have the following key estimate.

\begin{lemma}\label{apriorigalerkin}
Let $T^N$ be given by \eqref{TN} and take $r\in(0,1)$. Assume that $\beta > \max\{ 2r/(2-r),2r/(1-r) \}$ and that $\varepsilon$ and $\alpha$ are small and satisfy $T^N>0$.  Then, for all $t \leq T^N$ we have
\begin{equation}
E_n(t)+\varepsilon\int_0^t\int_\Omega(a\gamma\rho_n^{\gamma-2} + \delta\beta\rho_n^{\beta-2})|\nabla\rho_n|^2d\mathbf{x}ds \leq E(0)+\varepsilon^{1/2} R.\label{ineqreg}
\end{equation}

Also,
\begin{equation}
||\varepsilon^{1/2}\nabla\rho_n||_{L^2(\Omega\times(0,T))} + ||\varepsilon^2\rho_{nt}||_{L^r(\Omega\times(0,T))} + ||\varepsilon^3 \Delta\rho_n||_{L^r(\Omega\times(0,T))}\leq C\label{rhotDeltarho}
\end{equation}
where $C$ is a universal constant independent of $\varepsilon$, $\alpha$, $n$ and $N$.
\end{lemma}

\begin{proof}
First, we deduce an energy identity in a similar way as we did when deducing \eqref{difE}. 

Taking $\eta=\mathbf{u}_n$ in \eqref{galerkin} and using equations \eqref{rhotermsofu}, \eqref{Htermsofu} we have
\begin{flalign}
&\frac{d}{dt}\int_\Omega\left( \frac{1}{2}\rho_n|\mathbf{u}_n|^2 + \frac{a}{\gamma-1}\rho_n^\gamma +\frac{\delta}{\beta-1}\rho_n^\beta+\frac{1}{2}|\mathbf{H}_n|^2\right)d\mathbf{x}&\nonumber\\
&\hspace{15mm}+\int_\Omega (\mu |\nabla \mathbf{u}_n|^2 + (\lambda+\mu)(\text{div}\hspace{.5mm}\mathbf{u}_n)^2+\nu|\nabla\mathbf{H}_n|^2)d\mathbf{x}\nonumber\\
&\hspace{30mm}+\varepsilon\int_\Omega(a\gamma\rho_n^{\gamma-2}+\delta\beta\rho_n^{\beta-2})|\nabla\rho_n|^2 d\mathbf{x}\nonumber\\
&\hspace{45mm}+\int_\Omega\alpha \frac{J_\mathbf{y}}{\rho_n}g'(1/\rho_n)h(|\psi_n\circ \mathbf{Y}|^2)\text{div}\hspace{.5mm}\mathbf{u}_n d\mathbf{x}=0.&&\label{En1}
\end{flalign}

As $\rho_n$ is a solution of equation \eqref{rhotermsofu} with $\mathbf{u}=\mathbf{u}_n$ we have that
\[
\frac{\text{div}\hspace{.5mm}\mathbf{u}_n}{\rho_n}=\left(\frac{1}{\rho_n}\right)_t + \mathbf{u}_n\cdot\nabla\left(\frac{1}{\rho_n}\right)+\varepsilon\frac{\Delta \rho_n}{\rho_n^2}
\]
Now, from the coordinate change and the definition of $v_n=v_n(\mathbf{y},t)$ we have
\[
v_{nt}=\left(\frac{1}{\rho_n}\right)_t + \mathbf{u}_n^N\cdot\nabla\left(\frac{1}{\rho_n}\right).
\]
Thus,
\begin{align*}
&\int_\Omega\alpha \frac{J_\mathbf{y}}{\rho_n}g'(1/\rho_n)h(|\psi_n\circ \mathbf{Y}|^2)\text{div}\hspace{.5mm}\mathbf{u}_n d\mathbf{x}=\int_{\Omega_\mathbf{y}}\alpha g(v_n)_t\hspace{.5mm} h(|\psi_n|^2) d\mathbf{y} \\
&\hspace{20mm}+ \int_\Omega\alpha g'(1/\rho_n)h(|\psi_n\circ\mathbf{Y}|^2)J_\mathbf{y}\left( \varepsilon \frac{\Delta\rho_n}{\rho_n^2}+(\mathbf{u}_n^N-\mathbf{u}_n)\cdot\frac{\nabla\rho_n}{\rho_n^2} \right)d\mathbf{x}
\end{align*}
Now, using equation \eqref{Schotermsofu} we have that
\begin{align*}
\int_{\Omega_\mathbf{y}}\alpha g(v_n)_t\hspace{.5mm} h(|\psi_n|^2) d\mathbf{y}=\frac{d}{dt}\int_{\Omega_\mathbf{y}}\left( \frac{1}{2}|\nabla_\mathbf{y}\psi_n|^2 + \frac{1}{4}|\psi_n|^4 + \alpha g(v_n)h(|\psi_n|^2)\right)d\mathbf{y}.
\end{align*}

Gathering this information in \eqref{En1} we have
\begin{flalign*}
&\frac{d}{dt}\int_\Omega\left( \frac{1}{2}\rho_n\mathbf{u}_n + \frac{a}{\gamma-1}\rho_n^\gamma +\frac{\delta}{\beta-1}\rho_n^\beta+\frac{1}{2}|\mathbf{H}_n|^2\right)d\mathbf{x}\\
&\hspace{3mm}+\frac{d}{dt}\int_{\Omega_\mathbf{y}}\left( \frac{1}{2}|\nabla_\mathbf{y}\psi|^2 + \frac{1}{4}|\psi|^4 + \alpha g(v)h(|\psi|^2)\right)d\mathbf{y}&\\
&\hspace{6mm}+\int_\Omega (\mu |\nabla \mathbf{u}_n|^2 + (\lambda+\mu)(\text{div}\hspace{.5mm}\mathbf{u}_n)^2+\nu|\nabla\mathbf{H}_n|^2)d\mathbf{x}\\
&\hspace{10mm}+\varepsilon\int_\Omega(a\gamma\rho_n^{\gamma-2}+\delta\beta\rho_n^{\beta-2})|\nabla\rho_n|^2 d\mathbf{x}&\\
&\hspace{13mm}=\int_\Omega\alpha g'(1/\rho_n)h(|\psi_n\circ\mathbf{Y}|^2)J_\mathbf{y}\left( \varepsilon \frac{\Delta\rho_n}{\rho_n^2}+(\mathbf{u}_n^N-\mathbf{u}_n)\cdot\frac{\nabla\rho_n}{\rho_n^2} \right)d\mathbf{x}.&
\end{flalign*}

In order to estimate the right hand side of this identity we use a bootstrap argument as follows. First, recalling \eqref{JY}, we have that
\begin{equation}
|J_\mathbf{y}(t)|\leq \exp \left[C_N\left(t+\int_0^t||u_n(s)||_{H_0^1(\Omega)}^2ds\right)\right].
\end{equation}

Next, we assume that 
\begin{equation}
\mu\int_0^t||u_n(s)||_{H_0^1(\Omega)}^2ds \leq E(0)+\varepsilon^{1/2} R\label{bootstrap}
\end{equation}
for all $t\leq T^N$. This is certainly the case for $t$ small enough. Accordingly, the following calculations hold as long as \eqref{bootstrap} is satisfied.

With this in mind, using \eqref{gh} and Poincar\'{e}'s inequality, we have that
\begin{flalign*}
&\frac{d}{dt}\int_\Omega\left( \frac{1}{2}\rho_n\mathbf{u}_n + \frac{a}{\gamma-1}\rho_n^\gamma +\frac{\delta}{\beta-1}\rho_n^\beta+\frac{1}{2}|\mathbf{H}_n|^2\right)d\mathbf{x}\\
&\hspace{7mm}+\frac{d}{dt}\int_{\Omega_\mathbf{y}}\left( \frac{1}{2}|\nabla_\mathbf{y}\psi|^2 + \frac{1}{4}|\psi|^4 + \alpha g(v)h(|\psi|^2)\right)d\mathbf{y}&\\
&\hspace{14mm}+\int_\Omega (\mu |\nabla \mathbf{u}_n|^2 + (\lambda+\mu)(\text{div}\hspace{.5mm}\mathbf{u}_n)^2+\nu|\nabla\mathbf{H}_n|^2)d\mathbf{x}\\
&\hspace{21mm}+\varepsilon\int_\Omega(a\gamma\rho_n^{\gamma-2}+\delta\beta\rho_n^{\beta-2})|\nabla\rho_n|^2 d\mathbf{x}&\\
&\leq \alpha C e^{C_N(T^N+\mu^{-1}(E(0)+\varepsilon^{1/2} R))}\int_\Omega\left( \varepsilon |\Delta\rho_n|+\mu|\nabla \mathbf{u}_n|^2 + a\gamma\rho_n^{\gamma-2}|\nabla\rho_n|^2\right)d\mathbf{x}.&
\end{flalign*}

Taking \eqref{TN} into consideration we see that
\begin{align*}
&\alpha C e^{C_N(T^N+\mu^{-1}(E(0)+\varepsilon^{1/2} R))}\int_\Omega \left( \varepsilon |\Delta\rho_n|+\mu|\nabla \mathbf{u}_n|^2 + a\gamma\rho_n^{\gamma-2}|\nabla\rho_n|^2\right)d\mathbf{x}\\
&\hspace{10mm}\leq C\varepsilon^3\int_\Omega|\Delta\rho_n| d\mathbf{x} + C\varepsilon^2\int_\Omega \mu|\nabla \mathbf{u}_n|^2d\mathbf{x} + C\varepsilon^2\int_\Omega a\gamma\rho_n^{\gamma-2}|\nabla\rho_n|^2 d\mathbf{x},
\end{align*}
and thus, if $\varepsilon\leq \min\{ (2C)^{-1}, (2C)^{-1/2} \}$ we have that
\begin{flalign}
&\frac{d}{dt}E_n(t)+\varepsilon\int_\Omega(a\gamma\rho_n^{\gamma-2}+\delta\beta\rho_n^{\beta-2})|\nabla\rho_n|^2 d\mathbf{x}\leq C\varepsilon^3\int_\Omega |\Delta\rho_n|d\mathbf{x},&\label{bootstrap2}
\end{flalign}
for all $t\leq T^N$, and some constant $C>0$ independent of $\alpha$, $\varepsilon$, $n$ and $N$. In particular given $r>1$ we have that
\begin{align}
&||\sqrt{\rho}\mathbf{u}_n||_{L^\infty(0,T;L^2(\Omega))}^2+||\rho_n||_{L^\infty(0,T;L^\beta(\Omega))}^\beta+||\mathbf{u}_n||_{L^2(0,T;H_0^1(\Omega))}^2\nonumber\\
&\hspace{60mm}\leq E(0) + C(r)||\varepsilon^3 \Delta\rho_n||_{L^r(\Omega\times(0,T))}.\label{bootstrap1}
\end{align}

Regarding the right hand side of this inequality, we are going to use $L^p-L^q$ estimates on the parabolic equation \eqref{rhotermsofu} in order to bound appropriately the $L^r(\Omega\times(0,T))$-norm of $\Delta \rho_n$ (for any fixed $T\leq T^N$). Said $L^p-L^q$ estimates read
\begin{align}
&||\rho_{t}||_{L^p(0,T;L^q(\Omega))}+||\varepsilon\Delta\rho||_{L^p(0,T;L^{q}(\Omega))}\nonumber\\
&\hspace{30mm}\leq c(p,q)(||\rho_0||_{W^{2,q}(\Omega)}+||\text{div}(\rho \mathbf{u})||_{L^p(0,T;L^q(\Omega))}).\label{LpLq}
\end{align}
for any $1<p,q<\infty$. Taking $p=q:=r$ in \eqref{LpLq} and applying it to $\rho_n$ we have
\begin{align}
&||\varepsilon\Delta\rho_n||_{L^r(\Omega\times(0,T))}\nonumber\\
&\hspace{5mm}\leq c(r)(||\rho_0||_{W^{2,r}(\Omega)}+||\text{div}(\rho_n \mathbf{u}_n)||_{L^r(\Omega\times(0,T))})\nonumber\\
  &\hspace{5mm}\leq c(r)(||\rho_0||_{W^{2,r}(\Omega)}+||\mathbf{u}_n\cdot\nabla\rho_n ||_{L^r(\Omega\times(0,T))}+||\rho_n \text{div}\mathbf{u}_n||_{L^r(\Omega\times(0,T))})\label{Lr}
\end{align}

On the one hand,
\begin{equation*}
||\rho_n \text{div}\mathbf{u}_n||_{L^{2\beta/(\beta+2)}(\Omega)}\leq ||\rho_n||_{L^\beta(\Omega)}||\mathbf{u}_n||_{H_0^1(\Omega)},
\end{equation*}
and therefore
\begin{equation}
||\rho_n \text{div}\mathbf{u}_n||_{L^2(0,T;L^{2\beta/(\beta+2)}(\Omega))}\leq ||\rho_n||_{L^\infty(0,T;L^\beta(\Omega))}||\mathbf{u}_n||_{L^2(0,t;H_0^1(\Omega))},\label{rhodivu}
\end{equation}

On the other hand, we need to estimate $||\nabla\rho_n\cdot \mathbf{u}_n||_{L^r(\Omega\times(0,T))}$, and for this we need a good estimate on $\nabla \rho_n$. Such an estimate is provided by the following $L^p-L^q$ estimate on equation \eqref{rhotermsofu}, analogue to \eqref{LpLq}
\begin{equation}
||\varepsilon\nabla\rho||_{L^p(0,T;L^{q}(\Omega))}\leq c(p,q)(||\rho_0||_{W^{1,q}(\Omega)}+||\text{div}(\rho \mathbf{u})||_{L^p(0,T;W^{-1,q}(\Omega))}).\label{LpLqfrac}
\end{equation}
At this point we choose $q=2$ and leave $p$ to be chosen conveniently. In connection with \eqref{LpLqfrac} we have that
 \begin{equation}
||\varepsilon\nabla\rho_n||_{L^p(0,T;L^2(\Omega))}\leq c(p)(||\rho_0||_{H^1(\Omega)}+||\rho_n \mathbf{u}_n||_{L^p(0,T;L^2(\Omega))}).\label{epsnablarho}
\end{equation}

By Sobolev's embedding for any $p'\in[1,\infty)$ we have, since $\Omega\subseteq \mathbb{R}^2$, that
\[
||\mathbf{u}_n||_{L^{p'}(\Omega)}\leq c(p')||\mathbf{u}_n||_{H_0^1(\Omega)}.
\]

This implies that 
\begin{equation}
||\rho_n\mathbf{u}_n||_{L^2(0,T;L^{p'}(\Omega))}\leq c(p')||\rho_n||_{L^\infty(0,T;L^\beta(\Omega))}||\mathbf{u}_n||_{L^2(0,T;H_0^1(\Omega))},
\end{equation}
for any $p'<\beta$. Furthermore, we have that
\[
||\rho_n\mathbf{u}_n||_{L^\infty(0,T;L^{2\beta/(\beta+1)}(\Omega))}\leq ||\rho_n||_{L^\infty(0,T;L^\beta(\Omega))}||\sqrt{\rho_n}\mathbf{u}_n||_{L^\infty(0,T;L^2(\Omega))}.
\]
Now, for $2<p'<\beta$ we have
\begin{equation}
||\rho_n\mathbf{u}_n||_{L^2(\Omega)}\leq ||\rho_n\mathbf{u}_n||_{L^{2\beta/(\beta+1)}(\Omega)}^{1-\sigma}||\rho_n\mathbf{u}_n||_{L^{p'}(\Omega)}^\sigma
\end{equation}
where, $\tfrac{1}{2}=(1-\sigma)\tfrac{\beta+1}{2\beta}+\sigma\tfrac{1}{p'}$ and $\sigma\in(0,1)$. Consequently, taking $p=\frac{2}{\sigma}>2$ we obtain
\begin{align*}
||\rho \mathbf{u}||_{L^p(0,T;L^2(\Omega))}&\leq ||\rho_n\mathbf{u}_n||_{L^\infty(0,T;L^{2\beta/(\beta+1)}(\Omega))}^{1-\sigma}||\rho_n\mathbf{u}_n||_{L^2(0,T;L^{p'}(\Omega))}^\sigma\\
&\leq ||\rho_n||_{L^\infty(0,T;L^\beta(\Omega))}||\sqrt{\rho_n}\mathbf{u}_n||_{L^\infty(0,T;L^2(\Omega))}^{1-\sigma}||\mathbf{u}_n||_{L^2(0,T;H_0^1(\Omega))}^\sigma.
\end{align*}

In connection with \eqref{epsnablarho} we have that
\begin{align*}
&||\varepsilon\nabla\rho||_{L^p(0,T;L^2(\Omega))}\\
&\hspace{2mm}\leq c(p)(||\rho_0||_{H^1(\Omega)}+||\rho_n||_{L^\infty(0,T;L^\beta(\Omega))}||\sqrt{\rho_n}\mathbf{u}_n||_{L^\infty(0,T;L^2(\Omega))}^{1-\sigma}||\mathbf{u}_n||_{L^2(0,T;H_0^1(\Omega))}^\sigma).
\end{align*}

Finally, we see that we can choose $p'$ so that $r=p/2$ and we have
\begin{flalign*}
||\varepsilon\nabla\rho_n\cdot\mathbf{u}_n||_{L^r(\Omega\times(0,T))}^r&\leq \int_0^t||\varepsilon\rho_n||_{L^2(\Omega)}^r||\mathbf{u}_n||_{L^{2r/(2+r)}(\Omega)}ds\\
  &\leq C\int_0^t||\varepsilon\rho_n||_{L^2(\Omega)}^r||\mathbf{u}_n||_{H_0^1(\Omega)}^rds\\
  &\leq C\left(\int_0^t||\varepsilon\rho_n||_{L^2(\Omega)}^p ds\right)^{r/p}\left(\int_0^t||\mathbf{u}_n||_{H_0^1(\Omega)}^2 ds\right)^{1/2}.
\end{flalign*}
In this way we have
\begin{flalign}
&||\varepsilon\nabla\rho_n\cdot\mathbf{u}_n||_{L^r(\Omega\times(0,T))}\leq C||\varepsilon \nabla\rho_n||_{L^p(0,T;L^2(\Omega))}||\mathbf{u}_n||_{L^2(0,T;H_0^1(\Omega))}^{1/r}\nonumber\\
&\leq C(||\rho_0||_{H^1(\Omega)}+||\rho_n||_{L^\infty(0,T;L^\beta(\Omega))}||\sqrt{\rho_n}\mathbf{u}_n||_{L^\infty(0,T;L^2(\Omega))}^{1-\sigma}||\mathbf{u}_n||_{L^2(0,T;H_0^1(\Omega))}^\sigma)\times\nonumber\\
&\hspace{50mm}\times||\mathbf{u}_n||_{L^2(0,T;H_0^1(\Omega))}^{1/r}.&&\label{nablarhou}
\end{flalign}

Then, for $\beta$ large enough so that $\tfrac{2\beta}{2+\beta}>r$ (which is equivalent to $\beta>\tfrac{2r}{2-r}$) we have that
\begin{equation}
||\rho_n\text{div}\mathbf{u}_n||_{L^r(\Omega\times(0,T))}\leq C||\rho_n\text{div}\mathbf{u}_n||_{L^2(0,T;L^{2\beta/(2+\beta)})}.
\end{equation}

Putting this together with \eqref{bootstrap1}, \eqref{Lr}, \eqref{rhodivu} and \eqref{nablarhou} we have that
\begin{align*}
&||\varepsilon^3 \Delta\rho_n||_{L^r(\Omega\times(0,T))}\leq C\varepsilon^2||\rho_0||_{W^{2,r}(\Omega)}+C\varepsilon ||\rho_0||_{H^1(\Omega)}^2 \\
&\hspace{40mm}+ C\varepsilon (E(0) + ||\varepsilon^3\Delta\rho_n||_{L^r(\Omega\times(0,T))})^{\tfrac{1}{\beta}+\tfrac{1}{2}+\tfrac{1}{2r}},
\end{align*} 
and consequently, if $\beta$ is large enough so that $\tfrac{1}{\beta}+\tfrac{1}{2}+\tfrac{1}{2r}\leq 1$ (in other words if $\beta\geq 2r/(1-r)$) and $\varepsilon$ is small we have
\begin{align*}
||\varepsilon^3 \Delta\rho_n||_{L^r(\Omega\times(0,T))}&\leq C\varepsilon(\varepsilon||\rho_0||_{W^{2,r}(\Omega)} + ||\rho_0||_{H^1(\Omega)}^2 + E(0)+1).
\end{align*}
In order to conclude, we observe that this last inequality together with \eqref{bootstrap2} and \eqref{bootstrap1} reconfirms our bootstrap assumption \eqref{bootstrap}, and implies \eqref{ineqreg}.
\end{proof}

\subsection{Convergence of the Faedo-Galerkin approximations}

The uniform estimates from Lemma \ref{apriorigalerkin} permit us to iterate the fixed point argument a finite number of times to extend the local approximate solutions to the interval $[0,T]$ (provided that $T\leq T^N$). The next step in the proof of Theorem \ref{regsys} consists in passing to the limit as $n\to\infty$. We point out that the convergence in the terms concerning $\rho_n$ and $\mathbf{u}_n$ can be justified similarly as in \cite[Section 7.3.6]{Fe} and the terms involving $\mathbf{H}_n$ may be treated as in \cite[Section 4]{HW}. Regarding the terms involving $\psi_n$ a direct application of Aubin-Lions Lemma (Lemma \ref{AubinLions}) yields the desired result. The details are as follows.

Let $N$, $\varepsilon$, $\alpha$ and $\delta$ be fixed, $0<T<T^N$ and $\{(\rho_n,\mathbf{u}_n,\mathbf{H}_n,\psi_n)\}_{n=1}^\infty$ be the approximate solution to the regularized system, defined in the time interval $[0,T]$, given by the Faedo-Galerkin method described above.

First, as $\rho_n$ satisfies \eqref{rhotermsofu}, we have that
\[
||\nabla \rho_n||_{L^2(\Omega\times(0,T))}\leq C(\varepsilon),
\]
for some constant that depends on $\varepsilon$, but is independent of $n$. This can be easily deduced by multiplying \eqref{rhotermsofu} by $\rho_n$ and integrating by parts. Using \eqref{rhotDeltarho} and \eqref{ineqreg}, Aubin-Lions Lemma \ref{AubinLions} implies that $\rho_n$ has a subsequence (not relabelled) such that
\begin{equation}
\rho_n\to \rho \text{ in } L^\beta(\Omega\times(0,T)).\label{rhontorho}
\end{equation}

Furthermore, by \eqref{ineqreg} we can assume that
\begin{equation}
\mathbf{u}_n\to \mathbf{u} \text{ weakly in } L^2(0,T;H_0^1(\Omega)).\label{untou}
\end{equation}

Next, we see that $\mathbf{H}_n$ satisfies the following equation, equivalent to \eqref{Htermsofu}, 
\begin{equation}
\begin{cases}
\mathbf{H}_t - \nabla \times(\mathbf{u}\times\mathbf{H})= \nu\Delta \mathbf{H}, &\text{ on } \Omega\times(0,T)\\
\text{div}\hspace{1mm}\mathbf{H}=0,&\text{ on } \Omega\times(0,T)\\
H=0, &\text{ on }\partial \Omega\\
H=H_0, &\text{ on }\{ t=0\}\times\Omega.
\end{cases}\label{Htermsofu'}
\end{equation}

Consequently, by \eqref{ineqreg} we can also use Aubin-Lions Lemma in order to conclude that (selecting a subsequence if necessary)
\begin{equation}
\mathbf{H}_n\to \mathbf{H} 
\end{equation}
strongly in $L^2(\Omega\times(0,T))$ and weakly(-*) in $L^2(0,T;H^1(\Omega))\cap L^\infty(0,T;L^2(\Omega))$. Furthermore, $\mathbf{H}$ satisfies
\[
\text{div}\mathbf{H}=0.
\] 

Now, from \eqref{ineqreg} and using the embedding we see that $\rho_n\mathbf{u}_n$ is uniformly bounded in $L^\infty(0,T;L^{m_\infty}(\Omega))$, where $m_\infty=\frac{2\gamma}{\gamma+1}$. Indeed,
\[
\int_\Omega |\rho_n\mathbf{u}_n|^{m_\infty} d\mathbf{x}\leq \left(\int_\Omega \rho_n |\mathbf{u}_n|^2d\mathbf{x}\right)^{1/2} \left(\int_\Omega \rho_n^\gamma d\mathbf{x}\right)^{1/\gamma}\leq C.
\]

Thus, as the convergence in \eqref{rhontorho} is strong we may assume that
\begin{equation}
\rho_n \mathbf{u}_n \to \rho \mathbf{u} \text{ weakly-* in } L^\infty(0,T;L^{m_\infty}(\Omega)).
\end{equation}

By the same token, we have that
\begin{equation}
(\nabla\times\mathbf{H}_n)\times\mathbf{H}_n\to (\nabla\times\mathbf{H})\times\mathbf{H},\label{HntoH}
\end{equation}
weakly in $L^1(\Omega\times(0,T))$, and
\begin{equation}
\nabla(\mathbf{u}_n\times\mathbf{H}_n)\to \nabla(\mathbf{u}\times\mathbf{H}),
\end{equation}
in the sense of distributions.

Next, in view of \eqref{Schotermsofu} Aubin-Lions lemma also yields
\begin{equation}
\psi_n\to \psi
\end{equation}
strongly in $C(0,T;L^2(\Omega))$ and weakly-* in $L^\infty(0,T;H_0^1(\Omega))$. 

Let us state (without proof) the following result, which is a consequence of the Ascoli-Arzel\`{a} theorem (see \cite[Corollary 2.1]{Fe}).
\begin{lemma}\label{coro2.1Fe}
Let $\overline{O}\subseteq \mathbb{R}^M$ be compact and let $X$ be a separable Banach space. Assume that $v_n:\overline{O}\to X^*$, $n=1,2,...$ is a sequence of measurable functions such that 
\[
ess \sup_{y\in\overline{O}}||v_n(y)||_{X^*}\leq C \hspace{5mm}\text{ uniformly in }n=1,2,...
\]
Moreover, let the family of (real) functions
\[
\langle v_n,\Phi\rangle :y\to\langle v_n(y),\Phi\rangle, \hspace{10mm}y\in\overline{O},n=1,2...
\]
be equi-continuous for any fixed $\Phi$ belonging to a dense subset in the space $X$.

Then, $v_n\in C(\overline{O};X_{weak}^*)$ for any $n=1,2,...$ and there exist $v\in C(\overline{O};X_{weak}^*)$ such that
\[
v_n\to v \text{ in } C(\overline{O};X_{weak}^*) \text{ as } n\to\infty,
\]
passing to a subsequence as the case may be.
\end{lemma}

In view of \eqref{galerkin} and using \eqref{ineqreg} we see that the functions
\[
t\to\int_\Omega \rho_n\mathbf{u}_n \eta^j d\mathbf{x}
\] 
form a precompact system in $C([0,T])$ for any fixed $j$. This implies, by Lemma \ref{coro2.1Fe} that in fact
\begin{equation}
\rho_n \mathbf{u}_n\to\rho\mathbf{u} \text{ in }C([0,T];L_{weak}^{2\gamma/(\gamma+1)}(\Omega)).\label{rhountorhou}
\end{equation}

A similar argument shows that the mapping
\[
t\to \int_\Omega\mathbf{H}\varphi d\mathbf{x}
\]
is continuous for any test function $\varphi$.

Now, as $\gamma>1$, $L_{weak}^{2\gamma/(\gamma+1)}(\Omega)$ is compactly embedded into $H^{-1}(\Omega)$ and, consequently,
\begin{equation}
\rho_n\mathbf{u}_n\otimes \mathbf{u}_n \to \rho \mathbf{u}\otimes\mathbf{u}\label{rhouxuntorhouxu}
\end{equation}
weakly in $L^2(0,T;L^{c_2}(\Omega))$, where $c_2=2\gamma/(\gamma+1)>1$.

Next, as $\rho_n$ and $\rho$ are strong solutions of \eqref{rhotermsofu}, we have that 
\[
||\rho_n(t)||_{L^2(\Omega)}^2+2\varepsilon \int_0^t||\nabla \rho_n||_{L^2(\Omega)}^2ds = -\int_0^t\int_\Omega \rho_n^2\text{div}\mathbf{u}_n d\mathbf{x}ds + ||\rho_0||_{L^2(\Omega)}^2,
\]
and 
\[
||\rho(t)||_{L^2(\Omega)}^2+2\varepsilon \int_0^t||\nabla \rho||_{L^2(\Omega)}^2ds = -\int_0^t\int_\Omega \rho^2\text{div}\mathbf{u} d\mathbf{x}ds + ||\rho_0||_{L^2(\Omega)}^2
\]

Using \eqref{rhontorho} and \eqref{untou} we see that the right hands side of the former converges to its counterpart in the latter and thus, 
\[
||\nabla \rho_n||_{L^2(\Omega\times(0,T))}^2\to ||\nabla \rho||_{L^2(\Omega\times(0,T))}^2,
\]
and
\[
||\rho_n(t)||_{L^2(\Omega)}^2\to||\rho(t)||_{L^2(\Omega)}^2
\]
for any $t\in[0,T]$, which implies the strong convergence 
\[
\nabla\rho_n \to \nabla \rho \textbf{ in }L^2(\Omega\times (0,T)).
\]

With this we conclude that 
\[
\nabla \mathbf{u}_n\cdot\nabla\rho_n\to \nabla\mathbf{u}\cdot\nabla \rho
\]
in the sense of distributions.

Finally, recalling the definition of $\mathbf{u}_n^N$ through \eqref{uN}, we note that the weak convergence in \eqref{untou} implies the strong convergence
\[
\mathbf{u}_n^N\to \mathbf{u}^N
\]
which implies that the sequence Jacobians of the Lagrangian transformation $J_{\mathbf{y}n}$ defined through $\mathbf{u}_n^N$ converge strongly to the corresponding one related to $\mathbf{u}^N$.

With this we have shown that equations \eqref{appE2rho}-\eqref{appE2Scho} are satisfied in the sense  of distributions (equation \eqref{galerkin} can be verified by taking test functions of the form $\psi(t)\eta_j(x)$, where $\psi\in C_0^\infty(0,T)$) by the limit function $(\rho,\mathbf{u},\mathbf{H},\psi)$ as each term appearing on those equations is the limit in the sense of distributions of the respective terms corresponding to the Faedo-Galerkin approximation $(\rho_n,\mathbf{u}_n,\mathbf{H}_n,\psi_n)$. We have also shown that the initial and boundary conditions \eqref{appE20}, \eqref{appE2bound} are satisfied in the sense of distributions. 

Lastly, inequality \eqref{regnergy} is a consequence of \eqref{ineqreg} and this completes the proof of Theorem \ref{regsys}.

\section{Vanishing artificial viscosity and interaction coefficients}\label{limit1}

Theorem \ref{regsys} guarantees the existence of solutions to the Short Wave-Long Wave Interactions regularized system \eqref{appE2rho}-\eqref{appE2Scho}. Our next goal is to show that the sequence (or a subsequence) of solutions to this system converge to a global solution of the of the decoupled limit system when $(\varepsilon,\alpha,N,\delta)\to (0,0,\infty,0)$. In this Section we analyze the limit as $(\varepsilon,\alpha,N)\to (0,0,\infty)$, leaving $\delta>0$ fixed. As pointed out before, we can do all of of this as long as 
\begin{equation}
\left(\frac{\varepsilon^2}{\alpha}\right)^{1/C_N}\to \infty.\label{relepsalphN}
\end{equation}

In order to achieve this, we essentially adapt the arguments in \cite[Section 7.4]{Fe} and in \cite{HW}. 

The key point in the argument is to show that the sequence of densities converges strongly, in order to account for the nonlinearites from the pressure terms in the momentum equation \eqref{app2u}. This is not straightforward, as it was in the previous section, since we loose regularity of the density as $\varepsilon\to 0$. In particular, an argument like that of Aubin-Lions lemma does not apply. In this direction, we can exploit the weak continuity properties of the effective viscous flux $p(\rho)-(\lambda+2\mu)\text{div}\mathbf{u}$, originally discovered by D.~Hoff (\cite{Ho}) and  P.-L. Lions (\cite{Li}).

Let us point out that the terms involving the velocity field, the magnetic field and the wave function can be treated essentially as in the previous Section. Regarding the strong convergence of densities, the proof of weak continuity of the effective viscous flux found in \cite{HW} (cf. \cite{Fe}) can be adapted with no major difficulties once we realize that \eqref{regnergy}, \eqref{JY}, \eqref{TN} and \eqref{gh} imply that the extra term, due to the SW-LW interactions, appearing in the momentum equation
\[
\alpha \nabla(\frac{J_{\mathbf{y}}}{\rho}g'(1/\rho)h(\psi|^2))
\]
tends to zero in the sense of distributions as $(\varepsilon,\alpha,N)\to(0,0,\infty)$ satisfying \eqref{relepsalphN}. Accordingly, and to avoid the overload of notation, we may assume that $N$ and $\alpha$ tend to $\infty$ and $0$ respectively as functions of $\varepsilon$ and denote by $(\rho_\varepsilon,\mathbf{u}_\varepsilon,\mathbf{H}_\varepsilon,\psi_\varepsilon)$ the solution of the regularized system provided by Theorem \ref{regsys}.

The plan is as follows. First we show that $\rho_\varepsilon$ is uniformly (in $\varepsilon$) bounded  in $L_{loc}^{\beta+1}(\Omega\times(0,T))$ so that we can ensure that $\delta\rho^{\beta}$ and $a\rho^\gamma$ have (weakly) convergent subsequences. We know from Theorem \ref{regsys} that $\rho_\varepsilon\in L^{\beta+1}(\Omega\times(0,T))$ for each $\varepsilon$, but we have not yet shown that they are uniformly bounded in this space. 

Second, we prove the continuity of the effective viscous flux. And finally, we use this last result in order to show that $\overline{\rho \log\rho}=\overline{\rho}\log \overline{\rho}$ where the over line stands for a weak limit of the sequence indexed by $\varepsilon$. This last bit of information is enough to conclude the strong convergence of the densities due to the following result, which we state without proof (see \cite[Theorem 2.11]{Fe}).
\begin{lemma}\label{teo2.11}
Let $O\subseteq \mathbb{R}^N$ be a measurable set and $\{ \mathbf{v}_n\}_{n=1}^\infty$ a sequence of functions in $L^1(O;\mathbb{R}^M)$ such that
\[
\mathbf{v}_n\to\mathbf{v} \text{ weakly in }  L^1(O;\mathbb{R}^M).
\]
Let $\Phi:\mathbb{R}^M\to (-\infty,\infty]$ be a lower semi-continuous convex function such that $\Phi(\mathbf{v}_n)\in L^1(O)$ for any $n$ and
\[
\Phi(\mathbf{v}_n)\to\overline{\Phi(\mathbf{v})} \text{ weakly in } L^1(O).
\]

Then,
\[
\Phi(\mathbf{v})\leq \overline{\Phi(\mathbf{v})} \text{ a.a. on } O.
\]

If, moreover, $\Phi$ is strictly convex on an open convex set $U\subseteq \mathbb{R}^M$, and
\[
\Phi(\mathbf{v})=\overline{\Phi(\mathbf{v})} \text{ a.a. on }O,, 
\]
then,
\[
\mathbf{v}_n(\mathbf{y})\to\mathbf{v}(\mathbf{y}) \text{ for a.e. } \mathbf{y}\in\{ \mathbf{y}\in O:\mathbf{v}(\mathbf{y})\in U\},
\]
extracting a subsequence as the case may be.
\end{lemma}

From this point on, $T>0$ will denote an arbitrary prefixed time and $C>0$ will be a constant that may change from line to line being independent of $\varepsilon$, $\alpha$ and $N$. We also assume that $\delta>0$ is fixed and that $(\varepsilon,\alpha,N)\to(0,0,\infty)$ satisfying \eqref{relepsalphN}. Accordingly, we can also assume that $(\rho^\varepsilon,\mathbf{u}^\varepsilon,\mathbf{H}^\varepsilon,\psi^\varepsilon)$ are all defined in the time interval $[0,T]$ and satisfy \eqref{regnergy}.

\subsection{Higher integrability of the density}

This subsection is devoted to the proof of the following estimate.

\begin{lemma}\label{rhobeta+1}
For any compact $O\subseteq (\Omega\times(0,T))$ there is a constant $c=c(O)$ independent of $\varepsilon$ (and $\alpha$ and $N$) such that
\begin{equation}
\delta\int_O \rho^{\beta+1}d\mathbf{x}\leq c(O).\label{estrhobeta+1}
\end{equation}
\end{lemma}

Before going through the proof, let us introduce some preliminaries.

As in \cite{FeNP,Fe,HW} we consider the operator $\mathcal{A}$ by its coordinates
\begin{equation}
\mathcal{A}_j[v]:=\Delta^{-1}[\partial_{x_j}v],\hspace{10mm}j=1,2,\label{operatorA}
\end{equation}
where $\Delta^{-1}$ stands for the inverse of the Laplacian in $\mathbb{R}^2$. Equivalently, $\mathcal{A}_j$  can be defined through its Fourier symbol as
\[
\mathcal{A}_j[v]=\mathcal{F}^{-1}\left[ \frac{-i\hspace{.3mm}\xi_j}{|\xi|^2}\mathcal{F}[v] \right], \hspace{10mm}j=1,2.
\]

As shown in \cite{Fe} the operator $\mathcal{A}$ has the following properties:
\begin{align}
&||\mathcal{A}_jv||_{W^{1,s}(\Omega)}\leq c(s,\Omega)||v||_{L^s(\mathbb{R}^2)}, &1<s<\infty,\label{regA1}
\end{align}
and consequently, by Sobolev's embeddings
\begin{align}
&||\mathcal{A}_jv||_{L^q(\Omega)}\leq c(s,\Omega)||v||_{L^s(\mathbb{R}^2)}, &&q\text{ finite, provided }\frac{1}{q}\geq \frac{1}{s}-\frac{1}{2},\label{regA2}\\
&||\mathcal{A}_jv||_{L^\infty(\Omega)}\leq c(s,\Omega)||v||_{L^s(\Omega)}, &&\text{if }s>2.\label{regA3}
\end{align}

Let us also introduce the following standard smoothing operator
\begin{equation}
[v]_\mathbf{x}^\omega(\mathbf{z}):=(\vartheta_\omega*v)(\mathbf{z})=\int_{\mathbb{R}^2}\vartheta_\omega(\xi-\mathbf{z})v(\xi)d\xi,\label{mollifier}
\end{equation}
where, for each $\omega>0$,
\[
\vartheta_\omega(\mathbf{z}):=\frac{1}{\omega^2}\vartheta\left( \frac{|\mathbf{z}|}{\omega} \right),\hspace{10mm}\mathbf{z}\in\mathbb{R}^2,
\]
and $\vartheta\in C_0^\infty((-1,1))$ with
\[
\vartheta(-\tau)=\vartheta(\tau),\hspace{10mm}\int_{\mathbb{R}^2}\vartheta(|\mathbf{z}|)d\mathbf{z} = 1,\hspace{10mm}\vartheta \text{ nonincreasing on } [0,\infty).
\]

Let us also observe that from \eqref{regnergy} we have, in particular, that
\begin{equation}
\rho_\varepsilon \text{ is bounded in }L^\infty(0,T;L^\beta(\Omega)),\label{rhobetaeps},
\end{equation}
\begin{equation}
\mathbf{u}_\varepsilon \text{ is bounded in }L^2(0,T;H_0^1(\Omega)).\label{uH01eps},
\end{equation} 
\begin{equation}
\mathbf{H}_\varepsilon \text{ is bounded in }L^\infty(0,T;L^2(\Omega))\cap L^2(0,T;H^1(\Omega)).\label{HH01eps}
\end{equation}
\begin{equation}
\psi_\varepsilon \text{ is bounded in }L^\infty(0,T;L^4(\Omega)\cap H_0^1(\Omega)).\label{psiH01eps}
\end{equation}

\begin{proof}[Proof of Lemma \ref{rhobeta+1}]
For $\omega>0$, set
\[
B_\omega=[\rho_\varepsilon]_{\mathbf{x}}^\omega.
\]
Let us recall that $\rho_\varepsilon$ and $\mathbf{u}_\varepsilon$ satisfy \eqref{appE2rho} a.a. on $\Omega\times(0,T)$, along with the boundary condition $(\nabla\rho_\varepsilon\cdot\mathbf{n})|_{\partial \Omega} =0$. Then, extending $\rho_\varepsilon$ and $\mathbf{u}_\varepsilon$ to be zero outside of $\Omega$ we have that
\begin{equation}
\rho_{\varepsilon t}+\text{div}(\rho_\varepsilon \mathbf{u}_\varepsilon)=\varepsilon\text{div}(\mathbbm{1}_{\Omega}\nabla\rho_\varepsilon)\label{rhoR2}
\end{equation}
in the sense of distributions in $\mathbb{R}^2\times(0,T)$, where $\mathbbm{1}_\Omega$ is the characteristic function of $\Omega$. 

Applying the smoothing operator $[\cdot]_\mathbf{x}^\omega$ to equation \eqref{rhoR2} we have
\begin{equation}
B_{\omega t}=f_\omega,\label{Bomegat}
\end{equation}
with
\[
f_\omega =  - \text{div}([\rho_\varepsilon \mathbf{u}_\varepsilon]_\mathbf{x}^\omega) + \varepsilon\text{div}[\mathbbm{1}_\Omega\nabla\rho_\varepsilon]_\mathbf{x}^\omega.
\]

Note that $h_\omega$ is uniformly bounded in $L^2(0,T;H^{-1}(\Omega))$. 

As in \cite{Fe} we choose the test function\footnote{Let us recall that our two dimensional model can be regarded as the three dimensional one under the assumption that the involved functions are independent of the third variable. In particular, the velocity field takes values in $\mathbb{R}^3$. Accordingly, in order to use $\varphi$ as a test function we define its third component as being identically equal to zero.}
\[
\varphi(\mathbf{x},t)=\zeta(t)\eta(x)\mathcal{A}[\xi(\cdot)B_\omega(\cdot,t)](\mathbf{x},t),
\]
where $\eta,\xi\in C_0^\infty(\Omega)$ and $\zeta\in C_0^\infty((0,T))$,  and use it in the momentum equation \eqref{appE2u} to obtain
\begin{equation}
\int_0^T\int_\Omega \zeta\eta\xi(a\rho_\varepsilon^\gamma + \delta\rho_\varepsilon^\beta)B_\omega d\mathbf{x}ds =\int_0^T\int_\Omega \zeta\eta \mathbb{S}_\varepsilon:(\nabla\Delta^{-1}\nabla)[\xi B_\omega]d\mathbf{x}ds +  \sum_{j=1}^{9} I_j ,\label{sumIj}
\end{equation}
where, 
\[
\mathbb{S}_\varepsilon = \lambda(\text{div}\mathbf{u}_\varepsilon)\text{Id} + \mu(\nabla\mathbf{u}_\varepsilon+(\nabla \mathbf{u}_\varepsilon)^\top)
\]
is the viscous stress tensor, and
\begin{align*}
&I_1=\int_0^T\int_\Omega \zeta\mathbb{S}_\varepsilon\nabla \eta\cdot \mathcal{A}[\xi B_\omega]d\mathbf{x}ds,\\
&I_2=-\int_0^T\int_\Omega \zeta(a\rho_\varepsilon^\gamma+\delta\rho_\varepsilon^\beta)\nabla \eta\cdot\mathcal{A}[\xi B_\omega]d\mathbf{x}ds,\\
&I_3=-\int_0^T\int_\Omega\zeta(\rho_\varepsilon\mathbf{u}_\varepsilon\otimes \mathbf{u}_\varepsilon)\nabla\eta\cdot\mathcal{A}[\xi B_\omega]d\mathbf{x}ds\\
&I_4=- \int_0^T\int_\Omega \zeta \mathbf{u}_\varepsilon\cdot (\nabla\Delta^{-1}\nabla)[\xi B_\omega]\eta\rho_\varepsilon\mathbf{u}_\varepsilon d\mathbf{x}ds\\
&I_5=-\int_0^T\int_\Omega\zeta\eta (\nabla\times\mathbf{H}_\varepsilon)\times\mathbf{H}_\varepsilon\cdot\mathcal{A}[\xi B_\omega]d\mathbf{x}ds\\
&I_6=-\int_0^T\int_\Omega \zeta_t\hspace{.5mm}\eta\rho_\varepsilon\mathbf{u}_\varepsilon\cdot\mathcal{A}[\xi B_\omega]d\mathbf{x}ds\\
&I_7=-\int_0^T\int_\Omega\zeta\eta \rho_\varepsilon\mathbf{u}_\varepsilon\cdot\mathcal{A}[\xi f_\omega]d\mathbf{x}ds\\
&I_8=\varepsilon\int_0^T\int_\Omega\zeta\eta\nabla\mathbf{u}_\varepsilon\nabla\rho_\varepsilon\cdot\mathcal{A}[\xi B_\omega]d\mathbf{x}ds\\
&I_9=-\int_0^T\int_\Omega\zeta\alpha\frac{J_\mathbf{y}}{\rho_\varepsilon}g'(1/\rho_\varepsilon)h(|\psi_\varepsilon|^2)\big( \eta\xi B_\omega + \nabla\eta\cdot\mathcal{A}[\xi B_\omega] \big)d\mathbf{x}ds
\end{align*}

Note that by \eqref{regA3}, we have that
\begin{equation}
\mathcal{A}[\xi B_\omega] \text{ are bounded in }L^\infty(\Omega\times(0,T)),\label{AxiBinfty}
\end{equation}
provided that $\beta>2$. This together with \eqref{rhobetaeps} and \eqref{uH01eps} implies that the integrals $I_1$, $I_2$, $I_3$ and $I_7$ are bounded by a constant independent of $\varepsilon$ and $\omega$. Next, by \eqref{regA1} combined with \eqref{rhobetaeps} and \eqref{uH01eps} we have that $I_4$ is also bounded. Now, by the fact that $T\leq T_N$ combined with \eqref{TN}, \eqref{JY}, \eqref{regnergy} and \eqref{relepsalphN} we see that
\[
\alpha |J_\mathbf{y}|\leq \varepsilon^2,
\]
and thus, by \eqref{gh}, $I_9\to 0$ as $\varepsilon \to 0$. In particular, $I_9$ is also bounded by a constant independent of $\varepsilon$ and $\omega$. 

Regarding $I_7$, we see that $\rho_\varepsilon$, being a solution of equation \eqref{rhotermsofu}, satisfies the identity
\[
||\rho_\varepsilon(t)||_{L^2(\Omega)}^2+2\varepsilon \int_0^t||\nabla \rho_\varepsilon||_{L^2(\Omega)}^2ds = -\int_0^t\int_\Omega \rho_\varepsilon^2\text{div}\mathbf{u}_\varepsilon d\mathbf{x}ds + ||\rho_0||_{L^2(\Omega)}^2,
\]
and therefore we see that
\[
\varepsilon^{1/2}\nabla\rho_\varepsilon \text{ are uniformly bounded in }L^2(0,T;L^2(\Omega)).
\]

In particular, by \eqref{regA1}
\[
\mathcal{A}[\xi f_\varepsilon] \text{ are uniformly bounded in }L^2(\Omega\times(0,T)),
\]
Thus, we conclude that $I_7$ is bounded by a constant independent of $\varepsilon$ and $\omega$. By the same token we see that $I_8$ is uniformly bounded as well. In fact, we have that $I_8\to 0$ as $\varepsilon \to 0$.

Next, we see that \eqref{HH01eps} and \eqref{AxiBinfty} imply that $I_5$ is also bounded by a constant independent of $\varepsilon$ and $\omega$.

Finally, \eqref{regA1} and \eqref{uH01eps} also yield a uniform bound for the integral
\[
\int_0^T\int_\Omega \zeta\eta \mathbb{S}_\varepsilon:(\nabla\Delta^{-1}\nabla)[\xi B_\omega]d\mathbf{x}ds.
\]

Gathering all this information in \eqref{sumIj} and letting $\omega\to 0$ we arrive at \eqref{estrhobeta+1}. Of course, the bounds obtained for the integrals above depend on $\zeta$, $\eta$ and $\xi$, which is why the result is local.
\end{proof}

%----------------------------------------------------------------------------------------------------

\subsection{The effective viscous flux}

This section concerns the proof of the weak continuity of the effective viscous flux. However, before we get to it we have to make a few observations.

By \eqref{rhobetaeps}, \eqref{uH01eps}, \eqref{HH01eps} and \eqref{psiH01eps} we can assume \eqref{rhoepstorho}--\eqref{psiepstopsi}, where the strong convergence in \eqref{HepstoH} and in \eqref{psiepstopsi} is due to Aubin-Lions Lemma (Lemma \ref{AubinLions}).

Then, by the same arguments used to obtain \eqref{HntoH}, \eqref{rhountorhou} and \eqref{rhouxuntorhouxu} we see that
\begin{align}
&(\nabla\times \mathbf{H}_\varepsilon)\times \mathbf{H}_\varepsilon\to (\nabla\times \mathbf{H})\times\mathbf{H},\text{ in the sense of distributions},\label{derivHepstoH}\\
&\nabla\times (\mathbf{u}_\varepsilon\times\mathbf{H}_\varepsilon)\to\nabla\times (\mathbf{u}\times\mathbf{H})\text{ in the sense of distributions},\\
&\rho_\varepsilon \mathbf{u}_\varepsilon \to \rho\mathbf{u} \text{ in }C([0,T];L_{weak}^{2\beta/(\beta+1)}(\Omega)),\label{rhouepstorhou}\\
&\rho_\varepsilon \mathbf{u}_\varepsilon\otimes \mathbf{u}_\varepsilon\to\rho\mathbf{u}\otimes\mathbf{u} \text{ weakly in }L^2(0,T;L^{c_2}(\Omega)),\label{rhouuepstorhouu}
\end{align}
where, $c_2=2\gamma/(1+\gamma)>1$. 

As pointed out before we have that
\begin{equation}
\varepsilon \nabla\mathbf{u}_\varepsilon\cdot\nabla \rho_\varepsilon \to 0
\end{equation}
and
\begin{equation}
\alpha\nabla\left( \frac{J_\mathbf{y}}{\rho_\varepsilon}g'(1/\rho_\varepsilon)h(|\psi_\varepsilon|^2) \right) \to 0\label{alphacoupl}
\end{equation}
in the sense of distributions.  

Moreover, by \eqref{estrhobeta+1} we can assume that
\begin{equation}
a\rho^\gamma + \delta\rho^\beta \to \overline{p} \text{ weakly in }L^{(\beta+1)/\beta}(\Omega\times(0,T)).\label{pepstop}
\end{equation}

All of this information implies that the limit functions satisfy the equations
\begin{flalign}
&\rho_t + \text{div}(\rho\mathbf{u}) = 0&\label{rhoeps0}\\
&(\rho \mathbf{u})_t + \text{div}(\rho\mathbf{u}\otimes\mathbf{u}) + \nabla \overline{p} = \text{div}\mathbb{S} + \text{curl}\left(\mathbf{H} \right)\times\mathbf{H}.&\label{ueps0}
\end{flalign}
in the sense of distributions.

With this, we can state the result on the weak continuity of the effective viscous flux, originally discovered by P.-L. Lions (see \cite{Li}), as (cf. \cite{Fe,FeNP,HW})
\begin{lemma}\label{lemmaeffvisc}
Let $(\rho_\varepsilon,\mathbf{u}_\varepsilon,\mathbf{H}_\varepsilon,\psi_\varepsilon)$ be the solution of the regularized system provided by Theorem \ref{regsys}. Then,
\begin{align}
\lim_{\varepsilon\to 0}&\int_0^T\int_\Omega \zeta\eta (a\rho_\varepsilon^\gamma+\delta\rho_\varepsilon^\beta-(\lambda+2\mu)\text{div}\mathbf{u}_\varepsilon)\rho_\varepsilon d\mathbf{x}ds\nonumber\\
&=\int_0^T\int_\Omega \zeta\eta (a\overline{\rho^\gamma}+\delta\overline{\rho^\beta}-(\lambda+2\mu)\text{div}\mathbf{u})\rho d\mathbf{x}ds,\label{effviscflux}
\end{align}
for any $\zeta\in C_0^\infty((0,T))$, and $\eta\in C_0^\infty(\Omega)$.
\end{lemma}

\begin{proof}
First, noting that
\[
\xi\text{div}([\rho_\varepsilon\mathbf{u}_\varepsilon]_\mathbf{x}^\omega)=\text{div}(\xi[\rho_\varepsilon\mathbf{u}_\varepsilon]_\mathbf{x}^\omega) - \nabla\xi\cdot [\rho_\varepsilon\mathbf{u}_\varepsilon]_\mathbf{x}^\omega,
\]
we see that $I_7$ in \eqref{sumIj} may be rewritten as
\[
I_7=I_7^1+I_7^2+I_7^3,
\]
where
\begin{align*}
&I_7^1=\int_0^T\int_\Omega \zeta \xi [\rho_\varepsilon\mathbf{u}_\varepsilon]_\mathbf{x}^\omega (\nabla\Delta^{-1}\text{div})[\eta\rho_\varepsilon\mathbf{u}_\varepsilon]d\mathbf{x}ds\\
&I_7^2=-\int_0^T\int_\Omega \zeta\eta \rho_\varepsilon\mathbf{u}_\varepsilon\mathcal{A}\big[\nabla\xi\cdot[\rho_\varepsilon\mathbf{u}_\varepsilon]_\mathbf{x}^\omega\big] d\mathbf{x}ds\\
&I_7^3=-\varepsilon\int_0^T\int_\Omega \zeta \eta\rho_\varepsilon\mathbf{u}_\varepsilon\mathcal{A}[\xi\text{div}(\mathbbm{1}_\Omega \nabla\rho_\varepsilon)]d\mathbf{x}ds.
\end{align*}

Therefore, passing to the limit as $\omega\to 0$ in \eqref{sumIj} we obtain
\begin{align}
\int_0^T\int_\Omega \zeta\eta\big(\xi(a\rho_\varepsilon^\gamma + \delta\rho_\varepsilon^\beta)\rho_\varepsilon - \mathbb{S}_\varepsilon:(\nabla\Delta^{-1}\nabla)[\xi \rho_\varepsilon]\big)d\mathbf{x}ds =  \sum_{j=1}^{9} J_j^\varepsilon \nonumber\\
+\int_0^T\int_\Omega \zeta\mathbf{u}_\varepsilon\big( \xi\rho_\varepsilon(\nabla\Delta^{-1}\text{div})[\eta\rho_\varepsilon\mathbf{u}_\varepsilon]-(\nabla\Delta^{-1}\nabla)[\xi\rho_\varepsilon]\eta\rho_\varepsilon\mathbf{u}_\varepsilon \big)d\mathbf{x}ds,\label{sumJeps}
\end{align}
where, 
\begin{align*}
&J_1^\varepsilon=\int_0^T\int_\Omega \zeta\mathbb{S}_\varepsilon\nabla \eta\cdot \mathcal{A}[\xi \rho_\varepsilon]d\mathbf{x}ds,\\
&J_2^\varepsilon=-\int_0^T\int_\Omega \zeta(a\rho_\varepsilon^\gamma+\delta\rho_\varepsilon^\beta)\nabla \eta\cdot\mathcal{A}[\xi \rho_\varepsilon]d\mathbf{x}ds,\\
&J_3^\varepsilon=-\int_0^T\int_\Omega\zeta(\rho_\varepsilon\mathbf{u}_\varepsilon\otimes \mathbf{u}_\varepsilon)\nabla\eta\cdot\mathcal{A}[\xi \rho_\varepsilon]d\mathbf{x}ds\\
&J_4^\varepsilon=-\int_0^T\int_\Omega\zeta\eta (\nabla\times\mathbf{H}_\varepsilon)\times\mathbf{H}_\varepsilon\cdot\mathcal{A}[\xi \rho_\varepsilon]d\mathbf{x}ds\\
&J_5^\varepsilon=-\int_0^T\int_\Omega \zeta_t\hspace{.5mm}\eta\rho_\varepsilon\mathbf{u}_\varepsilon\cdot\mathcal{A}[\xi \rho_\varepsilon]d\mathbf{x}ds\\
&J_6^\varepsilon=-\int_0^T\int_\Omega \zeta\eta \rho_\varepsilon\mathbf{u}_\varepsilon\mathcal{A}[\nabla\xi\cdot\rho_\varepsilon\mathbf{u}_\varepsilon] d\mathbf{x}ds\\
&J_7^\varepsilon=-\varepsilon\int_0^T\int_\Omega \zeta \eta\rho_\varepsilon\mathbf{u}_\varepsilon\mathcal{A}[\xi\text{div}(\mathbbm{1}_\Omega \nabla\rho_\varepsilon)]d\mathbf{x}ds\\
&J_8^\varepsilon=\varepsilon\int_0^T\int_\Omega\zeta\eta\nabla\mathbf{u}_\varepsilon\nabla\rho_\varepsilon\cdot\mathcal{A}[\xi \rho_\varepsilon]d\mathbf{x}ds\\
&J_9^\varepsilon=-\int_0^T\int_\Omega\zeta\alpha\frac{J_\mathbf{y}}{\rho_\varepsilon}g'(1/\rho_\varepsilon)h(|\psi_\varepsilon|^2)\big( \eta\xi \rho_\varepsilon + \nabla\eta\cdot\mathcal{A}[\xi \rho_\varepsilon] \big)d\mathbf{x}ds
\end{align*}

Now, using equations \eqref{rhoeps0} and \eqref{ueps0}, a similar procedure yields
\begin{align}
\int_0^T\int_\Omega \zeta\eta\big(\xi(a\rho^\gamma + \delta\rho^\beta)\rho - \mathbb{S}:(\nabla\Delta^{-1}\nabla)[\xi \rho]\big)d\mathbf{x}ds =  \sum_{j=1}^{6} J_j \nonumber\\
+\int_0^T\int_\Omega \zeta\mathbf{u}\big( \xi\rho(\nabla\Delta^{-1}\nabla)[\eta\rho\mathbf{u}]-(\nabla\Delta^{-1}\text{div})[\xi\rho]\eta\rho\mathbf{u} \big)d\mathbf{x}ds,\label{sumJ}
\end{align}
where, 
\begin{align*}
&J_1=\int_0^T\int_\Omega \zeta\mathbb{S}\nabla \eta\cdot \mathcal{A}[\xi \rho]d\mathbf{x}ds,\\
&J_2=-\int_0^T\int_\Omega \zeta\overline{p}\nabla \eta\cdot\mathcal{A}[\xi \rho]d\mathbf{x}ds,\\
&J_3=-\int_0^T\int_\Omega\zeta(\rho\mathbf{u}\otimes \mathbf{u})\nabla\eta\cdot\mathcal{A}[\xi \rho]d\mathbf{x}ds\\
&J_4=-\int_0^T\int_\Omega\zeta\eta (\nabla\times\mathbf{H})\times\mathbf{H}\cdot\mathcal{A}[\xi \rho]d\mathbf{x}ds\\
&J_5=-\int_0^T\int_\Omega \zeta_t\hspace{.5mm}\eta\rho\mathbf{u}\cdot\mathcal{A}[\xi \rho]d\mathbf{x}ds\\
&J_6=-\int_0^T\int_\Omega \zeta\eta \rho\mathbf{u}\mathcal{A}[\nabla\xi\cdot\rho\mathbf{u}] d\mathbf{x}ds
\end{align*}

Following \cite{Fe,HW}, we now proceed to show that all the integrals in the right hand side of \eqref{sumJeps} converge to their counterparts in \eqref{sumJ}.

As $\rho_\varepsilon$ satisfies equation \eqref{appE2rho}, Lemma \ref{coro2.1Fe} yields
\begin{equation}
\rho_\varepsilon\to\rho \text{ in }C([0,T];L_{weak}^\beta(\Omega)),\label{rhoepstohocont}
\end{equation}
and consequently, by \eqref{regA1} and the compactness of the embedding $W^{1,\beta}(\Omega)\to C(\overline{\Omega})$ (recall that $\beta>2$) we have that
\[
\mathcal{A}[\xi\rho_\varepsilon]\to\mathcal{\xi\rho} \text{ in }C(\Omega\times(0,T)),
\]
Thus, in light of \eqref{uepstou}, \eqref{pepstop}, \eqref{derivHepstoH}, \eqref{rhouepstorhou} and \eqref{rhouuepstorhouu}, we have that 
\[
J_k^\varepsilon\to J_k, \text{ for }k=1,2,3,4,5.
\]

Similarly, by \eqref{rhoepstorho} and \eqref{uepstou} we have, in particular, that
\begin{equation}
\rho_\varepsilon \mathbf{u}_\varepsilon \text{ is bounded in }L^2(\Omega\times(0,T)),\label{rhouL2}
\end{equation}
and this together with \eqref{regA1} and \eqref{rhoepstohocont} implies that
\[
\mathcal{\nabla\xi\cdot\rho_\varepsilon\mathbf{u}_\varepsilon}\to\mathcal{\nabla\xi\cdot\rho\mathbf{u}} \text{ weakly in }L^2(0,T;H^1(\Omega)).
\]

Consequently, taking \eqref{rhouepstorhou} into account we have that
\[
J_6^\varepsilon\to J_6.
\]

As was already mentioned we have that
\[
J_k^\varepsilon\to 0, \text{ for }j=7,8,9.
\]

In order to deal with the last term on the right hand side of \eqref{sumJeps} we state the following result (see \cite[Corollary 6.1]{Fe}, also \cite[Lemma 3.4]{FeNP}).
\begin{lemma}\label{coro6.1Fe}
Let $O\subseteq \mathbb{R}^N$ be an arbitrary domain.
\begin{enumerate}
\item[(i)] Let 
\[
\mathbf{v}_n\to\mathbf{v} \text{ weaky in }L^p(O;\mathbb{R}^N),\hspace{5mm}\mathbf{w}_n\to\mathbf{w} \text{ weaky in }L^q(O;\mathbb{R}^N),
\]
with
\[
1<p,\hspace{5mm}, q<\infty,\hspace{5mm}\frac{1}{p}+\frac{1}{q}\leq 1.
\]
Then
\begin{align*}
&\mathbf{v}_n\cdot(\nabla\Delta^{-1}\text{div})[\mathbf{w}_n]-\mathbf{w}_n\cdot(\nabla\Delta^{-1}\text{div})[\mathbf{v}_n]\\
&\hspace{20mm}\to \mathbf{v}\cdot(\nabla\Delta^{-1}\text{div})[\mathbf{w}]-\mathbf{w}\cdot(\nabla\Delta^{-1}\text{div})[\mathbf{v}]
\end{align*}
in the sense of distributions.
\item[(ii)] Under the same hypotheses, if
\[
B_n\to B \text{ weakly in }L^p(O),\hspace{5mm}\mathbf{v}_n\to\mathbf{v}\text{ weakly in }L^q(O;\mathbb{R}^n),
\]
then
\[
(\nabla\Delta^{-1}\nabla)[B_n]\mathbf{v}_n - (\nabla\Delta^{-1}\text{div})[\mathbf{v}_n]B_n\to(\nabla\Delta^{-1}\nabla)[B]\mathbf{v} - (\nabla\Delta^{-1}\text{div})[\mathbf{v}]B
\]
\end{enumerate}
\end{lemma}
The proof of this result consists in applying a particular case of the celebrated Div-Curl Lemma (\cite{Mu,T,T'}). We refer to \cite{Fe} for the proof.

Now, by \eqref{rhoepstorho} and \eqref{rhouepstorhou}, a direct application of the above Lemma implies
\begin{align*}
&(\nabla\Delta^{-1}\nabla)[\xi\rho_\varepsilon(t)]\eta\rho_\varepsilon\mathbf{u}_\varepsilon(t)-\xi\rho_\varepsilon(t)(\nabla\Delta^{-1}\text{div})[\eta\rho_\varepsilon\mathbf{u}_\varepsilon(t)]\\
&\hspace{30mm}\to(\nabla\Delta^{-1}\nabla)[\xi\rho(t)]\eta\rho\mathbf{u}(t)-\xi\rho(t)(\nabla\Delta^{-1}\text{div})[\eta\rho\mathbf{u}(t)],
\end{align*}
weakly in $L^{2\beta/(\beta+3)}(\Omega)$ , for each fixed $t$.

As we know $L^q(\Omega)$ is compactly embedded in $H^{-1}(\Omega)$ for each $q>1$ (remember that our spatial domain is a bounded open subset of $\mathbb{R}^2$). In particular,
\begin{align*}
&(\nabla\Delta^{-1}\nabla)[\xi\rho_\varepsilon]\eta\rho_\varepsilon\mathbf{u}_\varepsilon-\xi\rho_\varepsilon(\nabla\Delta^{-1}\text{div})[\eta\rho_\varepsilon\mathbf{u}_\varepsilon]\\
&\hspace{30mm}\to(\nabla\Delta^{-1}\nabla)[\xi\rho]\eta\rho\mathbf{u}-\xi\rho(\nabla\Delta^{-1}\text{div})[\eta\rho\mathbf{u}],
\end{align*}
strongly in $L^2(0,T;H^{-1}(\Omega))$. As a consequence, keeping in mind \eqref{uepstou}, we see that
\begin{align*}
&\int_0^T\int_\Omega \zeta\mathbf{u}_\varepsilon\big( \xi\rho_\varepsilon(\nabla\Delta^{-1}\text{div})[\eta\rho_\varepsilon\mathbf{u}_\varepsilon]-(\nabla\Delta^{-1}\nabla)[\xi\rho_\varepsilon]\eta\rho_\varepsilon\mathbf{u}_\varepsilon \big)d\mathbf{x}ds\\
&\hspace{5mm}\to \int_0^T\int_\Omega \zeta\mathbf{u}\big( \xi\rho(\nabla\Delta^{-1}\nabla)[\eta\rho\mathbf{u}]-(\nabla\Delta^{-1}\text{div})[\xi\rho]\eta\rho\mathbf{u} \big)d\mathbf{x}ds.
\end{align*}

All of this information put together with \eqref{sumJeps} and \eqref{sumJ} yields
\begin{align}
&\lim_{\varepsilon\to 0}\int_0^T\int_\Omega \zeta\eta\big(\xi(a\rho_\varepsilon^\gamma + \delta\rho_\varepsilon^\beta)\rho_\varepsilon - \mathbb{S}_\varepsilon:(\nabla\Delta^{-1}\nabla)[\xi \rho_\varepsilon]\big)d\mathbf{x}ds\nonumber\\
&\hspace{15mm}=\int_0^T\int_\Omega \zeta\eta\big(\xi(a\rho^\gamma + \delta\rho^\beta)\rho - \mathbb{S}:(\nabla\Delta^{-1}\nabla)[\xi \rho]\big)d\mathbf{x}ds,\label{lastvisclim}
\end{align}
for any $\zeta\in C_0^\infty((0,T))$ and $\eta,\xi\in C_0^\infty(\Omega)$.

In order to conclude, as in \cite{Fe}, we compute
\begin{align}
&\int_0^T\int_\Omega\zeta \eta \mathbb{S}_\varepsilon:(\nabla\Delta^{-1}\nabla)[\xi \rho_\varepsilon]d\mathbf{x}ds\nonumber\\
&\hspace{15mm}=\int_0^T\int_\Omega\zeta \xi (\nabla\Delta^{-1}\nabla):(\eta\mathbb{S}_\varepsilon) \rho_\varepsilon d\mathbf{x}ds\nonumber\\
&\hspace{15mm}=\int_0^T\int_\Omega\zeta \xi (2\mu+\lambda)\text{div}(\eta\mathbf{u}_\varepsilon)\rho_\varepsilon d\mathbf{x}ds\nonumber\\
&\hspace{25mm}-\int_0^T\int_\Omega \zeta\xi \rho_\varepsilon[2\mu(\nabla\Delta^{-2}\nabla):(\mathbf{u}_\varepsilon\otimes\nabla\eta)+\lambda\mathbf{u}_\varepsilon\cdot\nabla\eta]d\mathbf{x}ds\nonumber\\
&\hspace{15mm}=\int_0^T\int_\Omega\zeta \xi \eta(2\mu+\lambda)\text{div}\mathbf{u}_\varepsilon\hspace{1mm}\rho_\varepsilon d\mathbf{x}ds\nonumber\\
&\hspace{25mm}-\int_0^T\int_\Omega 2\mu \zeta\xi \rho_\varepsilon[(\nabla\Delta^{-2}\nabla):(\mathbf{u}_\varepsilon\otimes\nabla\eta)-\mathbf{u}_\varepsilon\cdot\nabla\eta]d\mathbf{x}ds\label{lastvisceps}
\end{align}
and similarly
\begin{align}
&\int_0^T\int_\Omega\zeta \eta \mathbb{S}:(\nabla\Delta^{-1}\nabla)[\xi \rho]d\mathbf{x}ds\nonumber\\
&\hspace{15mm}=\int_0^T\int_\Omega\zeta \xi \eta(2\mu+\lambda)\text{div}\mathbf{u}\hspace{1mm}\rho d\mathbf{x}ds\nonumber\\
&\hspace{25mm}-\int_0^T\int_\Omega 2\mu \zeta\xi \rho[(\nabla\Delta^{-2}\nabla):(\mathbf{u}\otimes\nabla\eta)-\mathbf{u}\cdot\nabla\eta]d\mathbf{x}ds\label{lastvisc}
\end{align}

Taking \eqref{rhoepstohocont} into account, we see that the last integral on the right hand side of \eqref{lastvisceps} converges to the last integral in the right hand side of \eqref{lastvisc}. This and \eqref{lastvisclim} imply \eqref{effviscflux}, which concludes the proof.
\end{proof}

\subsection{Strong convergence of densities, renormalized solutions}\label{subsecstrongconv}

Using the results above we can show strong convergence of densities, essentially, in the same way as in \cite[Section 7.4.3]{Fe}. For this, we need to show first that the limit functions $\rho$ and $\mathbf{u}$ solve the continuity equation in the sense of renormalized solutions, that is, they satisfy \eqref{rhoeps0} in the sense of distributions, and more generally, \eqref{renormalized}, \eqref{renormB} and \eqref{renormBb}.

\begin{remark}\label{bnotbound}
The function $b$ in the definition of renormalized solutions does not have to be bounded. Indeed, provided that $\rho\in L^\infty(0,T;L^\gamma(\Omega))$ and $\mathbf{u}\in L^2(0,T;H_0^1(\Omega))$, by Lebesgue's dominated convergence theorem it can be shown that \eqref{renormalized} also holds for $b\in C[0,\infty)$ satisfying
\begin{equation}
|b'(z)z|\leq cz^{\gamma/2}, \text{ for } z \text{ larger than some positive constant }z_0.
\end{equation}
\end{remark}

Now, the fact that $\rho$ and $\mathbf{u}$ solve \eqref{rhoeps0} in the sense of renormalized solutions is a direct consequence of the following general result (cf. \cite[Corollary 4.1]{Fe})

\begin{lemma}\label{lemmarenormrhoL2}
Let $\Omega \subseteq \mathbb{R}^N$ be an arbitrary domain. Let,
\[
\rho\in L^2(\Omega\times(0,T))
\]
solve the continuity equation \eqref{rhoeps0} in the sense of distributions with
\[
\mathbf{u}\in L^2(0,T;H_0^1(\Omega)).
\]

Then, $\rho$ is a renormalized solution of \eqref{rhoeps0} on $\Omega\times (0,T)$.
\end{lemma}

This result follows by applying the the regularizing operator $v\to [v]_\mathbf{x}^\omega$ given by \eqref{mollifier} (that is, taking the functions $\vartheta_\omega$ as test functions) to equation \eqref{rhoeps0}, multiplying by $B'(\rho)$ and taking the limit as $\omega\to 0$, wherein the convergence is justified by the integrability properties of $\rho$ and $\mathbf{u}$ assumed as hypotheses. We omit the details.

Coming back to our present situation, as $\beta>2$ and by virtue of \eqref{rhoepstorho} and \eqref{uepstou} we can apply directly this result in order to conclude that $\rho$ and $\mathbf{u}$ indeed satisfy \eqref{renormalized}.

In particular, in view of Remark \ref{bnotbound} and using the fact that $\rho \in L^\infty(0,T;L^\beta(\Omega))$ we can choose $B(z)=z\log(z)$ in \eqref{renormalized} to conclude that the following equation is satisfied in the sense of distributions on $\mathbb{R}^2\times\Omega$:
\begin{equation}
(\rho\log(\rho))_t + \text{div}(\rho\log(\rho)\mathbf{u})+\rho\text{div}\mathbf{u}=0.\label{rhologrho}
\end{equation}

On the other hand, as $\rho_\varepsilon$ satisfies \eqref{appE2rho} a.e. on $\Omega\times(0,T)$, we can multiply \eqref{appE2rho} by $B'(\rho_\varepsilon)$ to obtain
\begin{align}
&B(\rho_\varepsilon)_t + \text{div}(B(\rho_\varepsilon)\mathbf{u}_\varepsilon)+\big( B'(\rho_\varepsilon)\rho_\varepsilon - B(\rho_\varepsilon) \big)\text{div}\mathbf{u}_\varepsilon \nonumber\\
&\hspace{30mm}= \varepsilon\text{div}(\mathbbm{1}_\Omega\nabla B(\rho_\varepsilon))-\varepsilon\mathbbm{1}_\Omega B''(\rho_\varepsilon)|\nabla\rho_\varepsilon|^2,
\end{align}
for any function $B\in C^2(\Omega)$ such that $B(0)=0$ with $B'$ and $B''$ uniformly bounded.

Accordingly, if $B$ is convex, and taking into account the boundary conditions \eqref{appE2bound}, we have
\[
\int_0^T \int_\Omega \zeta\big(B'(\rho_\varepsilon)\rho_\varepsilon-B(\rho_\varepsilon) \big) \text{div}\mathbf{u}d\mathbf{x}ds\leq \int_\Omega B(\rho_0)d\mathbf{x}+\int_0^T\int_\Omega \zeta_t B(\rho_\varepsilon)d\mathbf{x}ds,
\]
for any $\zeta\in C^\infty[0,T]$ with $\zeta(0)=1$ and $\zeta(T)=0$. 

Approximating the function $z\to z\log(z)$ by a sequence of convex functions $B$ as above we conclude that
\[
\int_0^T \int_\Omega \zeta\rho_\varepsilon \text{div}\mathbf{u}d\mathbf{x}ds\leq \int_\Omega \rho_0\log(\rho_0)d\mathbf{x}+\int_0^T\int_\Omega \zeta_t \rho_\varepsilon\log(\rho_\varepsilon) d\mathbf{x}ds.
\]

Taking the limit as $\varepsilon\to 0$ we obtain
\[
\int_0^T \int_\Omega \zeta\overline{\rho \text{div}\mathbf{u}}d\mathbf{x}ds\leq \int_\Omega \rho_0\log(\rho_0)d\mathbf{x}+\int_0^T\int_\Omega \zeta_t \overline{\rho\log(\rho)} d\mathbf{x}ds,
\]
where, as before, the over line stands for a weak limit of the sequence indexed by $\varepsilon$. In particular, by \eqref{rhoepstorho}, we can assume that $\rho_\varepsilon\log(\rho_\varepsilon)\to\overline{\rho\log(\rho)}$ weakly in $L^\infty(0,T;L^q(\Omega))$ for any $q<\beta$. As a consequence, 
\begin{equation}
\int_0^t\int_\Omega\overline{\rho \text{div}\mathbf{u}}d\mathbf{x}ds\leq \int_\Omega \rho_0\log(\rho_0)d\mathbf{x}+\int_\Omega  \overline{\rho\log(\rho)}(t) d\mathbf{x},\label{srong1}
\end{equation}
for any Lebesgue point $t$ of the function $\overline{\rho\log(\rho)}$.

Similarly, using a test function $\varphi(\mathbf{x},t)=\zeta(t)\eta(\mathbf{x})$ in \eqref{rhologrho}, where $\zeta$ and $\eta$ are smooth and $\zeta\geq 0$, $\eta\geq 0$, $\eta|_\Omega=1$, we obtain
\begin{equation}
\int_0^t\int_\Omega \rho \text{div}\mathbf{u} d\mathbf{x}ds = \int_\Omega \rho_0\log(\rho_0) d\mathbf{x} - \int_\Omega \rho\log(\rho)(t)d\mathbf{x},\label{strong2}
\end{equation}
for $t\in [0,T]$. Thus, from \eqref{srong1} and \eqref{strong2} we find the inequality
\begin{equation}
\int_\Omega \big(\overline{\rho\log(\rho)} - \rho\log(\rho)\big)(t) d\mathbf{x} \leq \int_0^t\int_\Omega \big( \rho \text{div}\mathbf{u} - \overline{\rho \text{div}\mathbf{u}}\big)  d\mathbf{x}ds,\label{rhologrhoesp}
\end{equation}
for a.e. $t\in[0,T]$.

Using Lemma \ref{lemmaeffvisc} we see that
\[
\int_O \big( \overline{\rho \text{div}\mathbf{u}} - \rho \text{div}\mathbf{u}\big)  d\mathbf{x}ds\geq \frac{1}{\lambda+2\mu}\liminf_{\varepsilon\to 0}\int_O \big( (a\rho_\varepsilon^{\gamma+1}+\delta\rho_\varepsilon^{\beta+1}) -\overline{p}\big)\rho d\mathbf{x}ds,
\]
for any compact $O\subseteq\Omega\times(0,T)$. Recall that 
\[
\overline{p}=a\overline{\rho^{\gamma}}+\delta\overline{\rho^{\beta}}.
\]

Now, as the function $z\to z^\beta$ is increasing we have 
\begin{align*}
\rho_\varepsilon^{\beta+1}-\overline{\rho^\beta}\hspace{1mm}\rho&=(\rho_\varepsilon^\beta-\rho^\beta)(\rho_\varepsilon-\rho)+\rho^\beta(\rho_\varepsilon-\rho)+(\rho_\varepsilon^\beta-\overline{\rho^\beta})\rho \\
&\geq \rho^\beta(\rho_\varepsilon-\rho) + (\rho_\varepsilon^\beta-\overline{\rho^\beta})\rho.
\end{align*}

Moreover, by virtue of Lemma \ref{rhobeta+1} we have that
\begin{align*}
&\rho_\varepsilon\to \rho \text{ weakly in } L^{\beta+1}(O),
&\rho_\varepsilon^\beta \to \overline{\rho^\beta} \text{ weakly in }L^{(\beta+1)/\beta}, 
\end{align*}
as $\varepsilon \to 0$. Thus, we conclude that
\begin{equation}
\liminf_{\varepsilon\to 0}\int_O \big(\delta\rho_\varepsilon^{\beta+1} -\delta\overline{\rho^\beta}\hspace{1mm}\rho\big) d\mathbf{x}ds \geq 0.
\end{equation}

By the same token, we have that
\begin{equation}
\liminf_{\varepsilon\to 0}\int_O \big(a\rho_\varepsilon^{\gamma+1} -a\overline{\rho^\gamma}\hspace{1mm}\rho\big) d\mathbf{x}ds \geq 0,
\end{equation}
and consequently, from \eqref{rhologrhoesp} we get
\begin{equation}
\int_\Omega \big(\overline{\rho\log(\rho)} - \rho\log(\rho)\big)(t) d\mathbf{x} \leq 0,\label{stongepsrhologrho}
\end{equation}
for a.e. $t$.

Finally, using Lemma \ref{teo2.11} we conclude that
\[
\overline{\rho\log(\rho)}=\rho\log(\rho),
\]
which, is equivalent to the strong convergence
\begin{equation}
\rho_\varepsilon \to \rho \text{ in }L^1(\Omega\times(0,T)) \text{ and a.e.}.
\end{equation}

In fact, by applying Lemma \ref{coro2.1Fe} we have \eqref{rhoepstorhoCL1}. 

In particular, we have that \eqref{rhoepsstrongtorho} in the sense of distributions.

\subsection{Conclusion}

With the strong convergence of the densities we have that all the nonlinearities present in the continuity and in the momentum equations are accounted for. Taking into account \eqref{rhoepstorho}-\eqref{alphacoupl} and also \eqref{rhoepstorhoCL1} and \eqref{rhoepsstrongtorho} we conclude that the limit functions $\rho$, $\mathbf{u}$, $\mathbf{H}$ and $\psi$ solve the  decoupled limit system \eqref{deltaE2rho}--\eqref{deltaE2Scho} with initial and boundary conditions \eqref{appE20} and \eqref{deltabound}, respectively, and we have proved Theorem~\ref{teoeps0}.

Let us recall that the regularized system \eqref{appE2rho}-\eqref{appE2Scho} was proposed as a regularized Short Wave-Long Wave interaction between the MHD System and the nonlinear Schr\"{o}dinger equation. Due to the lack of regularity of solutions, and in particular, due to the possible occurrence of vacuum in finite time, the Short Wave-Long Wave interactions could not be made in a straightforward way, as the Lagrangian transformation becomes singular in the presence of vacuum. To work around these difficulties we defined the Lagrangian coordinate through a smooth approximation $\mathbf{u}_N$ of the velocity field of the fluid, given by \eqref{uN}, and accordingly, by considering the limit as $N\to \infty$ satisfying \eqref{epsalphN}, Theorem \ref{teoeps0} serves the purpose to legitimize the coordinates of the limiting Schr\"{o}dinger equation to be considered as the Lagrangian coordinate in a generalized sense.

In short, we  have produced a finite-energy renormalized weak solution of the two dimensional MHD equations as a limit of solutions of the regularized Short Wave-Long Wave interactions.

Of course, there is one step left to complete the analysis, which consists in analysing the limit as $\delta\to 0$. Although the techniques are similar to those contained in this Section, there are a lot of limitations that have to be dealt with as we loose uniform boundedness of the sequence of densities in the space $L^\infty(0,T;L^\beta(\Omega))$. In particular, Lemma~\ref{lemmarenormrhoL2} can no longer be applied as we do not know, a priori, whether $\rho\in L^2(\Omega\times(0,T))$. Let us recall that $\beta$ was chosen conveniently large in order to justify the analysis developed.

Fortunately, we are now dealing with the {\em decoupled} system involving the two dimensional MHD equations and the nonlinear Schr\"{o}dinger equation, and the arguments in Section 5 of \cite{HW} can be followed literally line by line in order to justify the passing to the limit as $\delta\to 0$ in equations \eqref{deltaE2rho}-\eqref{deltaE2H'}. Finally, a simple application of Aubin-Lions Lemma (Lemma~\ref{AubinLions}) yields compactness of the sequence of solutions of \eqref{deltaE2Scho} as $\delta\to 0$.

In order to conclude we dedicate the following Section to quickly describe the passage to the limit as $\delta\to 0$ as in \cite[Section 5]{HW}.

\section{Vanishing artificial pressure}\label{limit2}

In the interest of analyzing the limit as $\delta\to 0$ we consider the limit problem \eqref{finalrho}--\eqref{finalSch} subject to initial and boundary conditions \eqref{final0} and \eqref{finalbound}.

Recall that we assume the initial data to be smooth in order to carry out the Faedo-Galerkin method from Section \ref{firstapp}. This constraint may be removed and we can consider more general initial data by means of approximation by smooth functions.

For system \eqref{finalrho}-\eqref{finalSch} above we consider initial data in \eqref{final0} satisfying \eqref{condsfinal0}.

Accordingly, we consider a sequence of approximate initial data denoted by $(\rho_{0\delta},\mathbf{u}_{0\delta},\mathbf{H}_{0\delta},\psi_{0\delta})$ such that
\begin{enumerate}
\item[(i)] \begin{align}
&\rho_{0\delta} \text{ is smooth and satisfies }\nabla\rho_{0\delta}\cdot\mathbf{n},\hspace{5mm} 0<\delta\leq \rho_{0\delta}\leq \delta^{-1/2\beta},\label{rhodelta-2beta}\\
&\rho_{0\delta}\to \rho_0 \text{ in }L^\gamma(\Omega), \hspace{5mm}|\{ x\in\Omega:\rho_{0\delta}<\rho_0 \}|\to 0,
\end{align}
as $\delta\to 0$.
\item[(ii)]\begin{equation}
\mathbf{m}_{0\delta}(\mathbf{x})=\begin{cases}
\mathbf{m}_0(\mathbf{x}), &\text{if }\rho_{0\delta}(\mathbf{x})\geq\rho_0(\mathbf{x}),\\
0, &\text{if }\rho_{0\delta}(\mathbf{x})<\rho_0(\mathbf{x}),
\end{cases}
\end{equation}
\item[(iii)] $\mathbf{H}_{0\delta}\to\mathbf{H}_0$ in $L^2(\Omega)$, and
\item[(iv)]$\psi_{0\delta}\to\psi_0$ in $H_0^1(\Omega)$.
\end{enumerate}

As aforementioned, once we have Theorem \ref{teoeps0}, the proof of Theorem \ref{teofinal} follows by repeating line by line the arguments in \cite[Section 5]{HW}. We refer to reader to \cite{HW} for the details.  The convergence of $\psi_\delta$ to a  solution of the cubic NLS follows trivially by Aubin-Lions lemma  (see \cite{JLi, Au}).

\end{document}